\documentclass[12pt]{article}
\usepackage[centertags]{amsmath}
\usepackage{amsfonts}
\usepackage{amssymb}
\usepackage{latexsym}
\usepackage{amsthm}
\usepackage{newlfont}
\usepackage{graphicx}
\usepackage{listings}
\usepackage{booktabs}
\usepackage{abstract}
\usepackage{enumerate}
\usepackage{xcolor}
\usepackage{accents} 

\usepackage{amsmath,amssymb,bbm}
\usepackage{amsfonts,amssymb}

\usepackage{mathtools}
\usepackage{enumerate}

\usepackage[numeric]{amsrefs}
\RequirePackage{srcltx}
\date{}

\usepackage[T1]{fontenc}

\usepackage[utf8]{inputenc}

\newlength{\defbaselineskip}
\newcommand{\setlinespacing}[1]%
           {\setlength{\baselineskip}{#1 \defbaselineskip}}

\newcommand{\N}{{\mathbb{N}}}

\newcommand{\actaqed}{\hfill $\actabox$}
{\medskip\noindent \textit{Proof of #1. }}%
{\actaqed \medskip}

\def\D{{\mathcal D}}

\def\cA{{\mathcal A}}
\def\C{{\mathcal C}}
\def\cC{{\mathcal C}}

\def\cZ{\mathcal{Z}}
\def \Tr{\mathcal T}

\def \cR{\mathcal R}

\def \cX{\mathcal X}

\def\R{{\mathbb R}}
\def\Z{\mathbb Z}

\def \T{\mathbb T}

\def\bbC{\mathbb C}
\def \<{\langle}
\def\>{\rangle}

\def \Og{\Omega}

\def \e{\varepsilon}

\def \ta{\theta}
\def \ff{\varphi}
\def\al{\alpha}

\def \ro{\varrho}

\def \sp{\operatorname{span}}

\def\bx{\mathbf x}
\def\by{\mathbf y}

\def\bk{\mathbf k}

\def\bw{\mathbf w}

\def\bs{\mathbf s}

\def\bX{\mathbf X}
\def\bW{\mathbf W}

\def\bF{\mathbf F}

\newtheorem{Theorem}{Theorem}[section]
\newtheorem{Lemma}{Lemma}[section]
\newtheorem{Definition}{Definition}[section]
\newtheorem{Proposition}{Proposition}[section]
\newtheorem{Remark}{Remark}[section]

\newtheorem{Corollary}{Corollary}[section]
\numberwithin{equation}{section}

\newcommand{\be}{\begin{equation}}
\newcommand{\ee}{\end{equation}}

\def\ta{\theta}
\def\al{{\alpha}}

\def\da{{\delta}}

\def\ld{{\lambda}}
\def\Bl{\Bigl}
\def\Br{\Bigr}
\def\f{\frac}

\def\vi{\varphi}

\def\k{{\kappa}}

\def\CB{{\mathcal B}}

\def\CC{{\mathbb C}}

\def\NN{{\mathbb N}}

\def\RR{{\mathbb R}}

\def\dist{\operatorname{dist}}

\def\Og{\Omega}

\def\Ld{\Lambda}

\def\sa{\sigma}

\def\cA{\mathcal{A}}

\def\cX{\mathcal{X}}

\def\bx{\mathbf{x}}

\def\by{\mathbf{y}}

\def\bw{\mathbf{w}}

\def\bx{\mathbf{x}}

\DeclareFontEncoding{FMS}{}{}
\DeclareFontSubstitution{FMS}{futm}{m}{n}
\DeclareFontEncoding{FMX}{}{}
\DeclareFontSubstitution{FMX}{futm}{m}{n}
\DeclareSymbolFont{fouriersymbols}{FMS}{futm}{m}{n}
\DeclareSymbolFont{fourierlargesymbols}{FMX}{futm}{m}{n}
\DeclareMathDelimiter{\VT}{\mathord}{fouriersymbols}{152}{fourierlargesymbols}{147}

\begin{document}

\title{Universal discretization and sparse sampling recovery}

\author{ F. Dai and   V. Temlyakov 	\footnote{
		The first named author's research was partially supported by NSERC of Canada Discovery Grant
		RGPIN-2020-03909.
		The second named author's research was supported by the Russian Science Foundation (project No. 23-71-30001)
at the Lomonosov Moscow State University.
  }}

\newcommand{\Addresses}{{
  \bigskip
  \footnotesize

  F.~Dai, \textsc{ Department of Mathematical and Statistical Sciences\\
University of Alberta\\ Edmonton, Alberta T6G 2G1, Canada\\
E-mail:} \texttt{fdai@ualberta.ca }

 \medskip
  V.N. Temlyakov, \textsc{University of South Carolina,\\ Steklov Institute of Mathematics,\\  Lomonosov Moscow State University,\\ and Moscow Center for Fundamental and Applied Mathematics.
  \\
E-mail:} \texttt{temlyak@math.sc.edu}

}}
\maketitle

\begin{abstract}
	  Recently, it was discovered that for a given function class $\bF$ the error of best linear recovery in the square norm can be bounded above by the Kolmogorov width of $\bF$ in the uniform norm. That analysis is based on deep results in discretization of the square norm of
functions from finite-dimensional subspaces. In this paper we show how very recent results on universal discretization of the square norm of functions from a collection of 
  finite-dimensional subspaces lead to an inequality (Lebesgue-type inequality) between the error of sparse recovery  in the square norm provided by the algorithm based on least squares operator and best sparse approximations in the uniform norm with respect to appropriate dictionaries. 
 
\end{abstract}


{\it MSC classification 2000:} Primary 65J05; Secondary 42A05, 65D30, 41A63.

\section{Introduction}
\label{I}

We discuss sampling recovery, when the error is measured in the $L_2$ norm. It is a very actively developing area of research. Recently an outstanding progress has been done in the sampling recovery in the $L_2$ norm (see, for instance, \cite{CM}, \cite{KU}, \cite{KU2}, \cite{NSU}, \cite{CoDo}, \cite{KUV}, \cite{TU1}, \cite{LimT}, \cite{JUV}, \cite{DKU}). In this paper we only discuss those known results, which are directly related to our new results. 
It is known that 
in this case (recovery in the $L_2$ norm) different versions of the classical least squares algorithms are very useful. In this paper 
we study the power of the classical least squares algorithms in the nonlinear (sparse) sample recovery setting.  This paper was motivated by the very recent paper by T. Jahn, T. Ullrich, and F. Voigtlaender \cite{JUV}. The authors of \cite{JUV} showed how some methods from compressed sensing 
allow us to prove an inequality between the error of optimal recovery and the error of sparse 
approximation with respect to a special system (for instance, the trigonometric system). 
The main point of this paper is to demonstrate that the recent results on discretization of the $L_2$ norm, namely, results on universal discretization allow us to prove similar inequalities for sparse recovery. We stress that in this paper we do not use any methods from compressed sensing. Instead we use recent results on universal discretization. 
The reader can find a survey of results on discretization in \cite{KKLT}. 

We now explain the setting in precise mathematical terms. Let $\Omega$ be a compact subset of $\R^d$ with the probability measure $\mu$. By the $L_p$ norm, $1\le p< \infty$, of the complex valued function defined on $\Omega$,  we understand
$$
\|f\|_p:=\|f\|_{L_p(\Omega,\mu)} := \left(\int_\Omega |f|^pd\mu\right)^{1/p}.
$$
By the $L_\infty$ norm we understand the uniform norm of continuous functions
$$
\|f\|_\infty := \max_{\bx\in\Omega} |f(\bx)|
$$
and with a little abuse of notations we sometimes write $L_\infty(\Omega)$ for the space $\C(\Omega)$ of continuous functions on $\Omega$.

 The concept of sparse approximation plays a fundamental role in our approach.  We begin our discussion with the concept of $v$-term approximation with respect to a given dictionary (system) $\D_N=\{g_i\}_{i=1}^N$. 
Given an integer $1\leq v\leq N$, we denote by $\mathcal{X}_v(\D_N)$ the collection of all linear spaces spanned by $g_j$, $j\in J$  with $J\subset [1,N]\cap \N$ and $|J|=v$, and 
denote by $\Sigma_v(\D_N)$  the set of all $v$-term approximants with respect to $\D_N$:
\begin{align*}
\Sigma_v(\D_N):&= \bigcup_{V\in\cX_v(\D_N)} V.
\end{align*}
We then define for a Banach space $X$
$$
\sigma_v(f,\D_N)_X := \inf_{g\in\Sigma_v(\D_N)}\|f-g\|_X
$$
to be the best $v$-term approximation of  $f\in X$   in the $X$-norm  with respect to $\D_N$,   and define 
$$
 \sigma_v(\bF,\D_N)_X := \sup_{f\in\bF} \sigma_v(f,\D_N)_X
 $$
  for a function class $\bF\subset X$. 
  
  For instance, in the case $X=L_2$ the best $v$-term approximation of $f$ with respect to 
  $\D_N$ can be realised by the following algorithm. Denote by $P_L$ the orthogonal $L_2$ 
  projection onto the subspace $L\subset L_2$ and define the algorithm
  $$
  L(f,v) :=  \text{arg}\min_{L\in \cX_v(\D_N)}\|f-P_L(f)\|_2,
$$
\be\label{I4a}
  P(\D_n,v)(f) := P_{L(f,v)}(f).
\ee
In this paper we study a similar algorithm with the operator $P_L$ replaced by the least squares operator, which we define momentarily. For a fixed set of points  $\xi:=\{\xi^j\}_{j=1}^m\subset \Omega$ define the sample vector
$$
S(f,\xi) := (f(\xi^1),\dots,f(\xi^m)) \in \bbC^m.
$$
Let $X_N$ be an $N$-dimensional subspace of the space of continuous functions $\C(\Omega)$ and let $\bw:=(w_1,\dots,w_m)\in \R^m$ be a positive weight, i.e. $w_j>0$, $j=1,\dots,m$. Consider the following classical weighted least squares recovery operator (algorithm) (see, for instance, \cite{CM})
$$
 \ell 2\bw(\xi,X_N)(f):=\text{arg}\min_{u\in X_N} \|S(f-u,\xi)\|_{2,\bw},\quad \xi=\{\xi^j\}_{j=1}^m\subset  \Omega,
$$
where
$$
\|S(g,\xi)\|_{2,\bw}:= \left(\sum_{\nu=1}^m w_\nu |g(\xi^\nu)|^2\right)^{1/2} .
$$
We note that the least squares operator (empirical risk minimization) is a classical algorithm from statistics and learning theory. It is also well known in learning theory that performance of this algorithm can be controlled by asymptotic characteristics measured in the uniform norm. The reader can find this kind of results and a discussion in \cite{VTbook}, Ch.4. 

In this paper we focus mainly on the case of special weights $\bw =\bw_m:= (1/m,\dots,1/m)$.   In this case the algorithm 
$\ell 2\bw_m(\xi,X_N)$ is the classical least squares algorithm. For brevity we use the notation
$$
LS(\xi,X_N) := \ell 2\bw_m(\xi,X_N).
$$
We define a new algorithm, which is a nonlinear algorithm:
$$
L(\xi,f) := \text{arg}\min_{ L\in \cX_v(\D_N)}\|f-LS(\xi,L)(f)\|_2,
$$
\be\label{I4}
  LS(\xi,\cX_v(\D_N))(f):= LS(\xi,L(\xi,f))(f).
\ee
Thus, the only difference between the algorithm $LS(\xi,\cX_v(\D_N))(\cdot)$ and the ideal 
one $P(\D_n,v)(\cdot)$ is that the orthogonal projection operator $P_L$ is replaced by the 
least squares operator $LS(\xi,L)$. In Section \ref{D} we give one more argument, which motivates us to use the least squares operator. 

{\bf Comment \ref{I}.1.} We now make a comment on the algorithm \newline $LS(\xi,\cX_v(\D_N))(\cdot)$ and on some known algorithms. At the step of finding $L(\xi,f)$ the algorithm 
$LS(\xi,\cX_v(\D_N))(\cdot)$ requires in addition to the sample vector $S(f,\xi)$ evaluation of 
the $L_2$ norm of a function. The authors of \cite{JUV} use the algorithm, which requires in addition to the sample vector $S(f,\xi)$ the fact that $f\in \bF$ and the value of the best $v$-term approximation with respect to $\D_N$ of the class $\bF$. Thus, both of the above algorithms require an extra information about $f$ in addition to the sample vector $S(f,\xi)$.
In the paper \cite{DT1} we address this important issue and use the greedy-type algorithm, namely, the Weak Orthogonal Greedy Algorithm, which only uses the sample vector $S(f,\xi)$ without any additional information about $f$. We point out that the algorithm $LS(\xi,\cX_v(\D_N))(\cdot)$ provides a $v$-term approximant with respect to the system $\D_N$.
The algorithm from \cite{JUV} does not provide a sparse approximant. 

We would like to understand how good  the new operator (approximation method) 
$LS(\xi,\cX_v(\D_N))(\cdot)$ is. Different criteria can be used for that. The classical approximation theory approach, which goes back to the concept of width (for more details see Section \ref{D}), suggests to compare accuracy of a given method on a function class 
$\bF$ with the optimal accuracy for the class $\bF$. There is a more delicate criterion, which 
was developed in nonlinear approximation. This criterion is based on the Lebesgue-type 
inequalities for a given method of approximation. It asks for the accuracy inequalities 
for all individual functions not for a given function class. Clearly, the criterion based on the Lebesgue-type inequalities is stronger than the one based on optimality for a function class. 
The reader can find a description of the Lebesgue-type inequalities approach in the book 
\cite{VTbookMA}, Ch.8. In this paper we apply the Lebesgue-type inequalities criterion
(see Theorem \ref{IT2} below) and as a corollary of the corresponding results we obtain 
results for function classes (see Theorem \ref{IT3} below). 

We now proceed to the formulation of our main results. We prove in Section \ref{A} two conditional theorems. We call these theorems {\it conditional} in order to emphasize that  the assumptions of these theorems are not easy to check. The condition on the $\D_N$ is formulated in terms of universal discretization. We now give the corresponding definitions. 

{\bf The  problem of universal discretization.} Let $\cX:= \{X(n)\}_{n=1}^k$ be a collection of finite-dimensional  linear subspaces $X(n)$ of the $L_p(\Omega,\mu)$, $1\le p < \infty$. We say that a set $\xi:= \{\xi^j\}_{j=1}^m \subset \Omega $ provides {\it universal discretization} for the collection $\cX$ if there are two positive constants $C_i$, $i=1,2$, such that for each $n\in\{1,\dots,k\}$ and any $f\in X(n)$ we have
$$
C_1\|f\|_p^p \le \frac{1}{m} \sum_{j=1}^m |f(\xi^j)|^p \le C_2\|f\|_p^p.
$$
 
 \begin{Definition}\label{ID1} We say that a set $\xi:= \{\xi^j\}_{j=1}^m \subset \Omega $ provides {\it one-sided universal discretization} with constant $C_1$ for the collection $\cX:= \{X(n)\}_{n=1}^k$ of finite-dimensional  linear subspaces $X(n)$ if we have
 \be\label{I3}
C_1\|f\|_2^2 \le \frac{1}{m} \sum_{j=1}^m |f(\xi^j)|^2\quad \text{for any}\quad f\in \bigcup_{n=1}^k X(n) .
\ee
We denote by $m(\cX,C_1)$ the minimal $m$ such that there exists a set $\xi$ of $m$ points, which
provides one-sided universal discretization with constant $C_1$ for the collection $\cX$. 
\end{Definition}
 
In this paper we focus on special collections of subspaces generated by a given dictionary:
$\cX = \cX_v(\D_N)$. In addition to a standard space $L_p(\Omega,\mu)$, $1\le p\le \infty$, we also use the following $L_p(\Omega,\mu_\xi)$ space, which is convenient for us.  For  $\xi=\{\xi^1,\ldots, \xi^m\}\subset\Omega$ let  $\mu_\xi$  denote the probability measure  
\[\mu_{\xi}=\frac 12 \mu+\frac1{2m}\sum_{j=1}^m \delta_{\xi^j},\]
where   $\delta_\bx$ denotes the Dirac measure supported at a point $\bx$. 

The following Theorem \ref{IT2} provides the Lebesgue-type inequalities for 
the algorithm $LS(\xi,\cX_v(\D_N))(\cdot)$. 
  
 \begin{Theorem}\label{IT2} Let $m$, $v$, $N$ be given natural numbers such that $v\le N$.  Let $\D_N\subset \C(\Og)$ be  a dictionary of $N$ elements. Assume that  there exists a set $\xi:= \{\xi^j\}_{j=1}^m \subset \Omega $, which provides {\it one-sided universal discretization} with constant $C_1$ for the collection $\cX_v(\D_N)$. Then for   any  function $ f \in \C(\Omega)$ we have
\be\label{I5}
  \|f-LS(\xi,\cX_v(\D_N))(f)\|_2 \le 2^{1/2}(2C_1^{-1} +1) \sigma_v(f,\D_N)_{L_2(\Og, \mu_\xi)}
 \ee
 and
 \be\label{I6}
  \|f-LS(\xi,\cX_v(\D_N))(f)\|_2 \le  (2C_1^{-1} +1) \sigma_v(f,\D_N)_\infty.
 \ee
 \end{Theorem}
 
  We now formulate a direct corollary of Theorem \ref{IT2} for function classes. Denote by
 $\cA(m,v,k,C_1)$   the family of all collections $\cX:= \{X(n)\}_{n=1}^k$ of finite-dimensional  linear subspaces $X(n)$, $\dim X(n)=v$, of the $L_2(\Omega,\mu)$ such that for each $\cX$  there exists a set $\xi:= \{\xi^j\}_{j=1}^m \subset \Omega $, which provides {\it one-sided universal discretization} with constant $C_1$ for the $\cX$ (see (\ref{I3}). Theorem \ref{IT4} provides some sufficient conditions to guarantee that the family $\cA(m,v,k,C_1)$ is nonempty. 
 We need some more notations. For $\bF\subset \cC(\Omega)$ denote
$$
 \sigma_v(\bF,\D_N)_{(2,m)}:= \sup_{\{\xi^1,\dots,\xi^m\}\subset\Omega} \sigma_v(\bF,\D_N)_{L_{2} (\Omega, \mu_{\xi})}.
$$
For brevity, in the case $X=L_p(\Omega,\mu)$ we write $\sigma_v(f,\D_N)_p$ instead of \newline $\sigma_v(f,\D_N)_{L_p(\Omega,\mu)}$.
For any dictionary $\D_N$ we have
$$2^{-1/2} \sigma_{v}(\bF,\D_N)_2\leq \sigma_{v}(\bF,\D_N)_{(2,m)}\leq  \sigma_v(\bF,\D_N)_\infty. $$

In this paper we study the following new
 characteristic, which controls performance of the algorithm $LS(\xi,\cX_v(\D_N))(\cdot)$ on 
 a function class $\bF$
$$
\varrho^{ls}_{m,v}(\bF,\D_N,L_2) := \inf_{\xi \in\Og^m} \sup_{f\in \bF} \min_{L\in\cX_v(\D_N)} \|f-LS(\xi,L)(f)\|_2.
$$
Clearly, for any integer  $m\ge 1$
$$
\varrho^{ls}_{m,v}(\bF, \D_N,L_2) \ge \sigma_v(\bF,\D_N)_2.
$$
The quantity $\varrho^{ls}_{m,v}(\bF,\D_N,L_2)$ shows how close we can get to the ideal $v$-term approximation error $\sigma_v(\bF,\D_N)_2$ by using function values at $m$ points and by applying associated least squares algorithms. 

\begin{Theorem}\label{IT3} Let $m$, $v$, $N$ be given natural numbers such that $v\le N$. Set $k=\binom{N}{v}$. Assume that a dictionary $\D_N$ is such that $\cX_v(\D_N) \in \cA(m,v,k,C_1)$. Then for   any compact subset $\bF$ of $\C(\Omega)$,  we have
\be\label{I7}
 \varrho_{m,v}^{ls}(\bF,\D_N,L_2(\Omega,\mu)) \le 2^{1/2}(2C_1^{-1} +1) \sigma_{v}(\bF,\D_N)_{(2,m)}
 \ee
 and
 \be\label{I8}
 \varrho_{m,v}^{ls}(\bF,\D_N,L_2(\Omega,\mu)) \le  (2C_1^{-1} +1) \sigma_{v}(\bF,\D_N)_\infty.
 \ee
 \end{Theorem}

Theorems \ref{IT2} and \ref{IT3} guarantee good performance of the algorithm \newline $LS(\xi,\cX_v(\D_N))(\cdot)$ under the one-sided universal discretization condition on the collection $\cX_v(\D_N)$. It turns out that recent results on universal discretization provide good 
unconditional bounds for performance of the algorithm $LS(\xi,\cX_v(\D_N))(\cdot)$ with respect to special systems $\D_N$. We now formulate the corresponding results. 
In Section \ref{B} we prove some unconditional results 
under certain assumptions on the dictionary $\D_N$. In order to apply Theorem \ref{IT3} we need 
to know the number $m(\cX_v(\D_N),C_1)$ defined in Definition \ref{ID1}. Here we use the known results on the universal discretization. In this paper we discuss in detail the following type of dictionaries. 

{\bf Example of a dictionary.} We assume  that 
$ \Phi_N:=\{\ff_j\}_{j=1}^N$ is a dictionary of $N$  uniformly bounded functions on $\Og \subset \R^d$ satisfying 
\be \label{I9}\sup_{\bx\in\Og} |\vi_j(\bx)|\leq 1,\   \ 1\leq j\leq N,\ee
and that there exists a constant $K>0$ such that   for any $(a_1,\ldots, a_N) \in\bbC^N,$
\begin{equation}\label{Riesz}
  \sum_{j=1}^N |a_j|^2 \le K  \left\|\sum_{j=1}^N a_j\ff_j\right\|^2_2 .
\end{equation}
Clearly, uniformly bounded (with bound $1$) orthonormal system, for instance, the trigonometric system, satisfies (\ref{I9}) and (\ref{Riesz}). 

 In \cite{DT} under an extra condition that $\Phi_N$ from the above example is a Riesz basis with constants $R_1$ and $R_2$
 we proved that
 \be\label{I11}
 m(\cX_v(\Phi_N),1/2) \le C(R_1,R_2)v(\log N)^2(\log(2v))^2.
 \ee  
 A proper modification of the corresponding proof from \cite{DT}, which we present in our 
 paper \cite{DT1} (see Theorem \ref{BT1} below), allows us to obtain a slightly better bound in the case of $\Phi_N$ from the above example
 \be\label{m}
 m(\cX_v(\Phi_N),1/2) \le C(K)	 v (\log N)  \log^2(2v)(\log (2v)+ \log\log N).
 \ee 

We formulate the corresponding corollary of Theorem \ref{IT3} as a theorem. 
 
 \begin{Theorem}\label{IT4} Let $ \Phi_N:=\{\ff_j\}_{j=1}^N$ be  a dictionary  satisfying \eqref{I9} and \eqref{Riesz} for some constant $K\ge 1$. 
 	Then there exists a constant $C=C_K\ge 1$ depending only on $K$ such that for any probability measure $\mu$ on the set $\Og\subset\RR^d$,  any compact subset $\bF$ of $\C(\Omega)$, and any integers $1\leq v\leq N$ and
 $$m\ge C	 v (\log N)  \log^2(2v)(\log (2v)+ \log\log N),$$
 	we have 
\be\label{I13}
 \varrho_{m,v}^{ls}(\bF,\Phi_N,L_2(\Omega,\mu) ) \le 5\sqrt{2} \sigma_{v} (\bF,\Phi_N)_{(2,m)}
 \ee
and 
\be\label{I14}
\varrho_{m,v}^{ls}(\bF,\Phi_N,L_2(\Omega,\mu) ) \le 5 \sigma_v(\bF,\Phi_N)_\infty.
\ee
\end{Theorem}

 
 As a corollary of Theorem \ref{IT4}, we obtain the following Corollary \ref{cor-1-1}.
 \begin{Corollary}\label{cor-1-1}
 	Let $ \Phi_N:=\{\ff_j\}_{j=1}^N$ be  a dictionary  satisfying \eqref{I9} and \eqref{Riesz} for some constant $K\ge 1$ and some  probability measure $\mu$ on  $\Og$.
  Let  
 	 \[ A_1^r(\Phi_N):=\Bigl\{ \sum_{ j=1}^N c_j \ff_j:\    \    \sum_{j=1}^N |c_j|j^r \leq 1\Bigr\}\]
 	 for some constant $r\ge 0$.
 	 Let $\bF\subset \cC(\Og)$ be a function class such that 
 	 \[ \dist(\bF, A_1^r(\Phi_N))_\infty :=\sup_{f\in\bF} \inf_{g\in A_1^r(\Phi_N)}\sup_{\bx\in\Og} |f(\bx)-g(\bx)| <\infty. \]
 	  Then there exists a constant $C_K\ge 1$ depending only on $K$ such that for any integers $1\leq v\leq N$ and
 	 $$m\ge C_K	 v (\log N)  \log^2(2v)(\log (2v)+ \log\log N),$$
 	we have 
 	\be\label{1-16}
 	\varrho_{m,2v}^{ls}(\bF,\Phi_N,L_2(\Omega,\mu) ) \le 5\sqrt{2}\left(   v^{-r-\f12}+  \dist(\bF, A_1^r(\Phi_N))_\infty\right)  .
 	\ee
 \end{Corollary}
The following comment is proved at the end of Section \ref{B}. 

{\bf Comment \ref{I}.2.}   We point out that in  the special  case when $\bF=A_1^r(\Phi_N)$ and the dictionary  $ \Phi_N:=\{\ff_j\}_{j=1}^N$ is orthogonal,  the estimate \eqref{1-16} in
 	   Corollary~\ref{cor-1-1}  is sharp  up to a general constant factor. 
	   
 In Section \ref{vt} we demonstrate how the above discussed technique of nonlinear sparse recovery, namely Theorem \ref{IT2}, 
 allows us to prove lower bounds for best $v$-term approximation in the uniform norm for special dictionaries. Also we prove there an inequality between linear and nonlinear asymptotic characteristics of some special function classes. Namely, we prove the following inequality 
 \be\label{I16}
 \kappa_m(W) \le (2C_1^{-1}+1) \sigma_v(W,\D_N)_\infty
 \ee
 provided that
 $m\ge m(\cX_v(\D_N), C_1).$
 Here,   $\kappa_m(W)$ is the best error of numerical 
 integration with $m$ knots of functions from $W$, and $W$ is a special function class (see detailed definitions in Section \ref{vt}).  
 
In Section \ref{C} we discuss in detail the case, when the dictionary $\D_N$ consists of trigonometric 
 functions $e^{i(\bk,\bx)}$ and the function class $\bF$ is a class of multivariate functions with mixed smoothness. Those results are based on Theorem \ref{IT4} and known results on best $v$-term approximation with respect to the trigonometric system. 
 
 In Section \ref{G} we study the univariate nonperiodic case. In addition to the results from 
 Sections \ref{I} -- \ref{B} we use there newly developed results on best $v$-term approximation with respect to the system of Gegenbauer polynomials.  These results extend and complement the corresponding results 
 from \cite{JUV}.

  \section{Conditional result}
\label{A}

We recall some notations and prove a conditional result similar to the one from \cite{VT183}. 
Let $X_N$ be an $N$-dimensional subspace of the space of continuous functions $\C(\Omega)$. For a fixed $m$ and a set of points  $\xi:=\{\xi^\nu\}_{\nu=1}^m\subset \Omega$ we associate with a function $f\in \C(\Omega)$ a vector (sample vector)
$$
S(f,\xi) := (f(\xi^1),\dots,f(\xi^m)) \in \bbC^m.
$$
Denote
$$
\|S(f,\xi)\|_p:= \left(\frac{1}{m}\sum_{\nu=1}^m |f(\xi^\nu)|^p\right)^{1/p},\quad 1\le p<\infty,
$$
and 
$$
\|S(f,\xi)\|_\infty := \max_{\nu}|f(\xi^\nu)|.
$$
For a positive weight $\bw:=(w_1,\dots,w_m)\in \R^m$ consider the following norm
$$
\|S(f,\xi)\|_{p,\bw}:= \left(\sum_{\nu=1}^m w_\nu |f(\xi^\nu)|^p\right)^{1/p},\quad 1\le p<\infty.
$$
Define the best approximation of $f\in L_p(\Omega,\mu)$, $1\le p\le \infty$ by elements of $X_N$ as follows
$$
d(f,X_N)_{L_p(\Og, \mu)} := \inf_{u\in X_N} \|f-u\|_{L_p(\Og, \mu)}.
$$
It is well known that there exists an element, which we denote $P_{X_N,p}(f)\in X_N$, such that
\begin{equation*}
\|f-P_{X_N,p}(f)\|_{L_p(\Og, \mu)} = d(f,X_N)_{L_p(\Og, \mu)}.
\end{equation*}
 (See, for instance, \cite[p. 59, Theorem 1.1]{DeL}).
The operator $P_{X_N,p}: L_p(\Omega,\mu) \to X_N$ is called the Chebyshev projection. 

We make the following two assumptions.

{\bf A1. Discretization.} Let $1\le p\le \infty$. Suppose that $\xi:=\{\xi^1,\ldots,\xi^m\}\subset \Omega$ is such that for any 
$u\in X_N$ in the case $1\le p<\infty$ we have
$$
C_1\|u\|_{L_p(\Og,\mu)} \le \|S(u,\xi)\|_{p,\bw}  
$$
and in the case $p=\infty$ we have
$$
C_1\|u\|_\infty \le \|S(u,\xi)\|_{\infty}  
$$
with a positive constant $C_1$ which may depend on $d$ and $p$. 

{\bf A2. Weights.} Suppose that there is a positive constant $C_2=C_2(d,p)$ such that 
$\sum_{\nu=1}^m w_\nu \le C_2$.

Consider the following well known recovery operator (algorithm) (see, for instance, \cite{CM})
$$
\ell p\bw(\xi)(f) := \ell p\bw(\xi,X_N)(f):=\text{arg}\min_{u\in X_N} \|S(f-u,\xi)\|_{p,\bw}.
$$
Note that the above algorithm $\ell p\bw(\xi)$ only uses the function values $f(\xi^\nu)$, $\nu=1,\dots,m$. In the case $p=2$ it is a linear algorithm -- orthogonal projection with respect 
to the norm $\|\cdot\|_{2,\bw}$. Therefore, in the case $p=2$ approximation error by the algorithm $\ell 2\bw(\xi)$ gives an upper bound for the recovery characteristic $\ro_m(\cdot, L_2)$. In the case $p\neq 2$ approximation error by the algorithm $\ell p\bw(\xi)$ gives an upper bound for the recovery characteristic $\ro_m^*(\cdot, L_p)$ (see Section \ref{D} below). 

 Under assumptions {\bf A1} and {\bf A2}, we have

\begin{Theorem}\label{AT1} Suppose that conditions  {\bf A1} and {\bf A2} are satisfied for some $\xi\in\Og^m$ and constants $C_1, C_2>0$.  Then for any $f\in \C(\Omega)$ and  $1\le p<\infty$ we have
\begin{equation}\label{A1}
\|f-\ell p\bw(\xi, X_N)(f)\|_{L_p(\Og, \mu)} \le 2^{1/p}(2C_1^{-1}C_2^{1/p} +1)d(f, X_N)_{L_p(\Og, \mu_{\bw,\xi})},
\end{equation}
where $\mu_{\bw, \xi}$ is  the probability measure  given by 
\[ \mu_{\bw, \xi}=\f12 \mu +\f 1{2\|\bw\|_1 }\sum_{j=1}^m w_j \da_{\xi^j},\   \ \|\bw\|_1=\sum_{j=1}^m w_j.\]
\end{Theorem}

Theorem \ref{AT1} was essentially proved in \cite{VT183}, where  the  result was formulated with  $d(f, X_N)_{\infty}$ in place of  $d(f, X_N)_{L_p(\Og, \mu_{\bw,\xi})}$. For the sake of completeness, we include a proof of Theorem \ref{AT1} below. 

\begin{proof}[Proof of Theorem \ref{AT1}]
	For simplicity, we write $u=\ell p\bw(\xi, X_N)(f)$. 
	For any $g\in X_N$,  we have
	\begin{align*}
	\|f&-u\|_{L_p(\Og, \mu)}\leq \|f-g\|_{L_p(\Og, \mu)}+\|g-u\|_{L_p(\Og, \mu)}\\
	&\leq \|f-g\|_{L_p(\Og, \mu)}+ C_1^{-1} \|S(g-u,\xi)\|_{p,\bw}  \\
	&\leq \|f-g\|_{L_p(\Og, \mu)}+  C_1^{-1}\Bl[ \|S(f-g,\xi)\|_{p,\bw}  +\|S(f-u,\xi)\|_{p,\bw} \Br]\\
	&\leq \|f-g\|_{L_p(\Og, \mu)}+  2C_1^{-1}\|S(f-g,\xi)\|_{p,\bw}\\
	& \leq 2^{1/p}(1+2C_1^{-1}C_2^{1/p} )\|f-g\|_{L_p(\Og, \mu_{\bw,\xi})},
	\end{align*}
	where we used assumption {\bf A1} and the fact that $g-u\in X_N$ in the second step, and the fact that 
	\[ u=\ell p\bw(\xi,X_N)(f):=\text{arg}\min_{u\in X_N} \|S(f-u,\xi)\|_{p,\bw}\]
	in the fourth step.
	In the last step we used  assumption {\bf A2} and the inequalities
$$
\|h\|_{L_p(\Og, \mu)}	\le 2^{1/p} \|h\|_{L_p(\Og,\mu_{\bw,\xi})},
$$
$$
\|S(h,\xi)\|_{p,\bw} \le (2 \|\bw\|_1)^{1/p}\|h\|_{L_p(\Og,\mu_{\bw,\xi})},
$$
which give
$$
\|f-g\|_{L_p(\Og, \mu)}\le 2^{1/p} \|f-g\|_{L_p(\Og,\mu_{\bw,\xi})}
$$
and
$$
\|S(f-g,\xi)\|_{p,\bw} \le (2 \|\bw\|_1)^{1/p}\|f-g\|_{L_p(\Og,\mu_{\bw,\xi})}.
$$
	Taking infimum over all $g\in X_N$ gives the desired estimate \eqref{A1}.
	
\end{proof}

\begin{Remark}\label{AR1} The proof similar to the above proof of Theorem \ref{AT1} gives 
the following inequality from \cite{VT183} instead of \eqref{A1}
\begin{equation}\label{A2}
\|f-\ell p\bw(\xi, X_N)(f)\|_{L_p(\Og, \mu)} \le (2C_1^{-1}C_2^{1/p} +1)d(f, X_N)_{\infty}.
\end{equation}
\end{Remark}

We now proceed to the proof of Theorem \ref{IT2}.

\begin{proof}[Proof of Theorem \ref{IT2}.]
Suppose that for $p=2$ a set $\xi:= \{\xi^j\}_{j=1}^m \subset \Omega $ provides one-sided universal discretization for the collection $\cX_v(\D_N)$ with constant $C_1$. Then condition {\bf A1} is satisfied for all $X(n)$ from the collection $\cX_v(\D_N)$ with $p=2$ and $\bw=\bw_m$. Clearly, condition {\bf A2} is satisfied with $C_2=1$. Thus, we can apply Theorem \ref{AT1} for each 
subspace $X(n)$ with the same set of points $\xi$. It gives for all $n=1,\dots,k$, $k=\binom{N}{v}$,
\be\label{A3}
\|f-LS(\xi,X(n))(f)\|_2 \le \sqrt{2}(2C_1^{-1} +1)d(f, X(n))_{L_2(\Og, \mu_\xi)},
\ee
where as above 
$$
\mu_\xi=\mu_{\bw_m, \xi}=\f12 \mu+\f 1{2m}\sum_{j=1}^m\da_{\xi^j}.
$$
 Then, inequality (\ref{A3}) and definition (\ref{I4}) imply
\be\label{A4}
 \|f-LS(\xi,\cX)(f)\|_2 \le \sqrt{2}(2C_1^{-1} +1)\min_{1\le n\le k}d(f, X(n))_{L_2(\Og, \mu_\xi)}.
 \ee
 This proves inequality (\ref{I5}) of Theorem \ref{IT2}. Inequality (\ref{I6}) follows from Remark \ref{AR1}.
 \end{proof} 
 
\section{Proofs of Theorem \ref{IT4} and Corollary \ref{cor-1-1}} 
\label{B}

The proof of Theorem \ref{IT4} is based on Theorem \ref{IT3} and the known result (see Theorem \ref{BT1} below) 
on universal discretization from \cite{DT} and \cite{DT1}. 
  We now cite a result from \cite{DT1} on universal discretization.   We consider a special case when  the dictionary  $\D_N=\Phi_N:=\{\ff_j\}_{j=1}^N$
  satisfies the conditions  \eqref{I9} and \eqref{Riesz}. 

 \begin{Theorem}[\cite{DT1}]\label{BT1} Let $1\le p\le 2$. Assume that $\Phi_N$ is a dictionary  satisfying \eqref{I9} and \eqref{Riesz} for some constant $K>0$. Then for a large enough constant $C=C(p,K)$, and any integer $1\leq v\leq N$,  there exist 
$m$ points
$\xi^1,\ldots, \xi^m \in  \Omega$ with
$$
m \le C v \log N\cdot  \log^2(2v)(\log (2v)+\log\log N)
$$
  such that for any $f\in  \Sigma_v(\Phi_N)$
we have
\begin{equation}\label{1.4u}
\frac{1}{2}\|f\|_p^p \le \frac{1}{m}\sum_{j=1}^m |f(\xi^j)|^p \le \frac{3}{2}\|f\|_p^p.
\end{equation}
\end{Theorem}

For given $N$, $v\le N$, and $K>0$,  denote by $\cR(N,v,K)$ the family of all collections $\cX_v(\Phi_N)$, where $\Phi_N$ is  a dictionary  satisfying \eqref{I9} and \eqref{Riesz}. Theorem \ref{BT1} guarantees that 
$\cR(N,v,K)$ is a subset of  $\cA(m,v,k,C_1)$ with $k=\binom{N}{v}$, $C_1=1/2$ and  $$m\ge Cv(\log N) (\log(2v))^2 (\log (2v)+\log\log N). $$

Note that for a dictionary $\D_N$ we have 
$$
\min_{L\in \cX_v(\D_N)} d(f,L)_p = \sigma_v(f,\D_N)_p,
$$
where $\sigma_v(f,\D_N)_p$ is the best $v$-term approximation of $f$ with respect to the dictionary $\D_N$ in the $L_p$-norm. 

Theorems \ref{IT3} and  \ref{BT1}   imply  Theorem \ref{IT4}. 
Theorems \ref{IT2} and  \ref{BT1}   imply the following  Theorem \ref{BT2}.

  \begin{Theorem}\label{BT2} Given a number $K>0$,  there exists a constant $C=C(K)$ such that for any   compact subset $\Omega$  of $\R^d$, any probability measure $\mu$ on $\Og$,   for any integer $1\leq v\leq N$, and any dictionary $\Phi_N$  satisfying both  \eqref{I9} and \eqref{Riesz},  there exists a set $\xi =\{\xi^\nu\}_{\nu=1}^m\subset \Omega$ such that  we have for each $f\in \cC(\Omega)$
\be\label{B2}
  \|f-LS(\xi,\cX_v(\Phi_N))(f)\|_2 \le 5\sqrt{2} \sigma_{v}(f,\Phi_N)_{L_2(\Omega,\mu_\xi)}
 \ee
 and
 \be\label{B2a}
  \|f-LS(\xi,\cX_v(\Phi_N))(f)\|_2 \le 5  \sigma_{v}(f,\Phi_N)_\infty,
 \ee
 provided
 $$m\ge Cv(\log N) (\log(2v))^2 (\log (2v)+\log\log N). $$
\end{Theorem}

{\bf Comment \ref{B}.1.} Theorem \ref{BT2} gives a version of the Lebesgue-type inequality (see, for instance,  \cite{VTbookMA}, Section 8.7) for the algorithm $LS(\xi,\cX_v(\Phi_N))$. The algorithm $LS(\xi,\cX_v(\Phi_N))$ provides a $v$-term approximant with respect to $\Phi_N$. 
The best accuracy that we can achieve with $v$-term approximants is $\sigma_v(f,\Phi_N)_2$. The proper 
Lebesgue-type inequality would give
\be\label{B3}
 \|f-LS(\xi,\cX_v(\Phi_N))(f)\|_2 \le C' \sigma_v(f,\Phi_N)_2. 
 \ee
 Theorem \ref{BT2} gives  weaker inequalites (\ref{B2}) and (\ref{B2a}) than (\ref{B3}) with $\sigma_v(f,\Phi_N)_2$ 
 replaced by $\sigma_{v}(f,\Phi_N)_{L_2(\Omega,\mu_\xi)}$ or  $\sigma_v(f,\Phi_N)_\infty$. It is known that for many function classes the quantity 
 $\sigma_v(\bF,\Phi_N)_\infty$ is bounded above by $C(\log (v+1))^{1/2}\sigma_v(\bF,\Phi_N)_2$. 
 For instance, it is the case for the trigonometric system and $\bF=\bW^r_q$ (see \cite{VTbookMA}, 
 Theorems 9.2.11 (p.462) and 9.2.13 (p.464)). 
 
 Finally, we turn to the proof of Corollary  \ref{cor-1-1}.
 
 \begin{proof}[Proof of Corollary  \ref{cor-1-1}]
 	
 	By Theorem  \ref{IT4} Corollary \ref{cor-1-1} follows from the following Proposition \ref{BP1}.
	
\begin{Proposition}\label{BP1} Let $\bF$ and  $A_1^r(\Phi_N)$ be from Corollary \ref{cor-1-1} and let $\Phi_N$ satisfy (\ref{I9}). Then for any $\xi\in\Og^m$ we have 
 	\be\label{3-4}
 	\sa_{2v}(\bF, \Phi_N)_{L_2(\Og, \mu_\xi)}\leq  v^{-r-\f12}+  \dist(\bF, A_1^r(\Phi_N))_\infty.
 	 	\ee
\end{Proposition} 
\begin{proof}		
 	 Given $f\in \bF$, take any  $f_0=\sum_{j=1}^N c_j\ff_j  \in A_1^r(\Phi_N)$.   	Let $f_1:=\sum_{v< j\leq N}  c_j\ff_j$. Then 
 	\be\label{vr}
	\sum_{v< j\leq N} |c_j| \leq  v^{-r} \sum_{j=1}^N |c_j| j^r \le  v^{-r}.
	\ee
 	Since   $L_2(\Og, \mu_\xi)$ is a Hilbert space and by (\ref{I9}) 
 	\be\label{ub}
	\|\vi_j\|_{L_2(\Og, \mu_\xi)}\leq \sup_{\bx\in\Og} |\vi_j(\bx)|\leq 1,\   \ 1\leq j\leq N,
	\ee
 	using Theorem 2.19 of \cite[p. 93]{VTbook}, we deduce from \eqref{vr}
 	 that  there exists a $g_1\in\Sigma_v(\Phi_N)$ (provided by a greedy algorithm)
 	such that
 	\[\|f_1-g_1\|_{L_2(\Og, \mu_\xi)}\leq  v^{-r-\f12}.\]
 	Now setting 
 	\[ g:=\sum_{1\leq j\le v} c_j \vi_j +g_1,\]
 	we have 
 	\begin{align*}
 	\sa_{2v}(f,\Phi_N)_{L_2(\Og, \mu_\xi)}&\leq \|f-f_0\|_\infty+\|f_0-g\|_{L_2(\Og, \mu_\xi)} \\
 	&=
 	 \|f-f_0\|_\infty+\|f_1-g_1\|_{L_2(\Og, \mu_\xi)}\\
 	 &\leq \|f-f_0\|_\infty+ v^{-r-\f12} .
 	\end{align*}
	Taking infimum over $f_0 \in  A_1^r(\Phi_N)$ we obtain that for any $\xi\subset \Omega$ 
	$$
	\sa_{2v}(f,\Phi_N)_{L_2(\Og, \mu_\xi)} \le \dist(f, A_1^r(\Phi_N))_\infty + v^{-r-\f12}.
	$$
 	It remains to take supremum over all $f\in\bF$ and obtain \eqref{3-4}.
 \end{proof}
 Corollary \ref{cor-1-1} is proved.	
 \end{proof}
 
 \begin{Remark}\label{BR1} Let $\Phi_N$ satisfy (\ref{I9}). Then bound (\ref{ub}) and Theorem 2.19 of \cite[p. 93]{VTbook} imply for any $\xi\subset \Omega$
 $$
 \sa_{v}(A_1(\Phi_N), \Phi_N)_{L_2(\Og, \mu_\xi)}\leq  v^{-\f12},\quad A_1(\Phi_N):=A_1^0(\Phi_N).
 $$
 It is interesting to compare this bound with the known one from \cite{DeT} (see Theorem 3.1 there). It is proved in \cite{DeT} that if  
  $\Phi_N$ satisfies both (\ref{I9}) and the discretization condition: There exists a subset  $\xi= \{\xi^1,\dots,\xi^M\}\subset \Omega$
 such that for any $f\in [\Phi_N]:= \sp(\ff_1,\dots,\ff_N)$ the inequality
 $$
 \|f\|_\infty \le K_2\max_{1\le\nu\le M}|f(\xi^\nu)|
 $$
 holds, then we have
 $$
 \sa_{v}(A_1(\Phi_N), \Phi_N)_{\infty}\leq  CK_2v^{-\f12}(\log (1+M/v))^{\f12}
 $$
 with an absolute constant $C$.
 \end{Remark}
 
 \begin{proof}[Proof of Comment  \ref{I}.3] 	   We use
 	    a powerful result of B. Kashin on lower estimates of $m$-term approximation by orthogonal system (see Lemma \ref{2-6-ch2} of this paper). There exist a
 	universal constant $C_0\in (0,1)$ and a constant $c_r>0$ depending only on $r$ such that for any orthonormal  system $\Phi_N=\{\vi_j\}_{j=1}^N$, and
 	any integers   $0<v\leq C_0 N$ and $m\ge 1$, we have
 	\[ 	\varrho_{m,v}^{ls}(A_1^r(\Phi_N),\Phi_N,L_2 ) \ge c_r v^{-r-\f12}.\]
  To see this, let $v_1$ be the integer such that $C_0^{-1} v\leq v_1<C_0^{-1} v+1$, and let
 	$$A_\infty(\Phi_N,v_1):=\left\{ \sum_{j=1}^{v_1} a_j \vi_j:\  \ |a_j|\leq 1,\   \ 1\leq j\leq  v_1\right\}.$$
 	If the dictionary $\Phi_N$ is orthonormal, then Kashin's result (Lemma \ref{2-6-ch2}) implies that
 	\[  \sigma_v(A_\infty(\Phi_N,v_1),\Phi_N)_2\ge \f {\sqrt{3} } 2  \sqrt{v}.\]
 On the other hand, 	it is easily seen  that $ c_r' v^{-r-1} A_\infty(\Phi_N,v_1)\subset A_1^r(\Phi_N)$ for some constant $c_r'>0$.  Thus,  we obtain
 	$$
 	\varrho_{m,v}^{ls}(A_1^r(\Phi_N),\Phi_N,L_2 ) \ge  \sigma_v(A_1^r(\Phi_N),\Phi_N)_2
	$$
	$$
	\ge c_r' v^{-r-1} \sigma_v(A_\infty(\Phi_N,v_1),\Phi_N)_2 \ge c_r v^{-r-\f12}.
 	$$
 \end{proof}

 \section{Lower bounds for the best $v$-term approximations}
 \label{vt}
 
 In this section we illustrate how Theorem \ref{IT2} can be used to prove lower bounds for the best 
 $v$-term approximations with respect to a special system $\D_N$. 
 
 {\bf \ref{vt}.1. Some general lower bounds.} We begin with some results for general dictionaries $\D_N$. For a set of points $\xi =\{\xi^1,\ldots,\xi^m\}\subset \Omega$ denote
 $$
 \cZ(\xi) := \{f\in \cC(\Omega)\,:\, f(\xi^\nu)=0,\quad \nu=1,\dots,m\}.
 $$
 Note that if $f\in  \cZ(\xi)$, then for each finite dimensional space $X\subset \cC(\Og)$, we have
 $ LS(\xi, X)(f) =0$. 
 As a result, we deduce from  Theorem \ref{IT2} the following corollary. 
 
   \begin{Theorem}\label{vtT1}  Let $v$ and $N$ be given natural numbers such that $v\le N$.  Let $\D_N\subset \C(\Og)$ be  a dictionary of $N$ elements. Take a number 
  \begin{equation}\label{vt1}
  m\ge m(\cX_v(\D_N),C_1) 
  \end{equation}
  with some positive constant $C_1$. Assume that a set $\xi:= \{\xi^j\}_{j=1}^m \subset \Omega $ provides one-sided universal discretization with constant $C_1$ for the collection $\cX_v(\D_N)$. 
    Then for   any  function $ f \in  \cZ(\xi)$ we have
\be\label{vt2}
 \|f\|_{L_2(\Omega,\mu)} \le  (2C_1^{-1} +1) \sigma_v(f,\D_N)_\infty.
 \ee
\end{Theorem}

Theorem \ref{vtT1} implies that for any function class $\bF\subset\cC(\Og)$, we have
\begin{equation}\label{vt3}
\tau_m(\bF)_2:=\inf_{\xi\in\Og^m} \sup_{f\in \bF\cap \cZ(\xi)} \|f\|_{L_2(\Og, \mu)} \leq (2C_1^{-1} +1) \sa_v (\bF,\D_N)_\infty
\end{equation}
whenever $m$ is a positive integer  satisfying \eqref{vt1}. Inequality (\ref{vt3}) indicates importance of the characteristic $\tau_m(\bF)_2$. We now prove a general lower bound for $\tau_m(\bF)_2$ under certain conditions on $\bF$. We begin with a conditional result. 

{\bf Condition D1}$(M,p)$. We say that an $D$-dimensional subspace $X_D$ of $\cC(\Og)$ satisfies condition {\bf D1}$(M,p)$ if there exists an $\bX=(\bx^1, \ldots, \bx^M)\in\Og^M$   such that  the inequalities 
	\be\label{vt4}
	\f 12 \|f\|_p^p \leq \|S(f,\bX)\|_p^p \leq \f 32 \|f\|_p ^p
	\ee
	and 
		\be\label{vt5}
		\f 12 \|f\|_2^2 \leq  \|S(f,\bX)\|_2^2\leq \f 32 \|f\|_2^2
		\ee
	hold simultaneously for all $f\in X_D$.   

\begin{Theorem}\label{vtT2} Let $2\le p<\infty$. Assume that $X_D$ satisfies condition {\bf D1}$(M,p)$. Let	$X_D^p:=\{ f\in X_D:\  \ \|f\|_p\leq 1\}$.
 Then  for any  integer  $1\leq m\leq D$ we have 
\begin{equation}\label{vt6}	
\tau_m(X_D^p)_2 \ge \frac{1}{3}\left(\frac{D-m}{M}\right)^{1/2}.
\ee
\end{Theorem}

{\bf Comment \ref{vt}.1.} We refer the reader to \cite{VTbookMA}, Section 3.2.6, p.85, for 
another powerful method of estimating $\tau_m(X_D^p)_2$, which is based on volume estimates.

For the proof of Theorem \ref{vtT2}, 
 we  need the following 
result of G.G.~Lorentz:
\begin{Lemma} \label{vtL1}\textnormal{\cite[Lemma 3.1, p.410]{LGM}}
	Given  an $n$-dimensional subspace $X$  of $\RR^m$,  there exists
	$(z_1, \dots, z_m)\in X$ such that $\max_{1\leq i\leq m} |z_i|=1$
	and $|\{ i:\  \  1\leq i\leq m, |z_i|=1\}|\ge n$.
\end{Lemma}

	Clearly,   Lemma \ref{vtL1} 	 implies  that for each $n$-dimensional complex  subspace $X$  of $\CC^m$,  there exists
	$(z_1, \dots, z_m)\in X$ such that $\max_{1\leq i\leq m} |z_i|=1$
	and $|\{ i:\  \  1\leq i\leq m, |z_i|\ge \f 1{\sqrt{2}}\}|\ge n$.

\begin{proof}[Proof of Theorem \ref{vtT2}]     
 
For an arbitrarily given    $\xi=(\xi^1, \ldots, \xi^m)\in\Og^m$ with $1\leq m\leq D$,  
	let $ V:= X_D\cap \cZ(\xi)$, and 
	$Y:=\{ S(f,\bX):\  \ f\in V\}$, where $\bX$ is from {\bf D1}$(M,p)$. Clearly,
	\[ \dim Y=\dim V\ge D-m .\]
Thus, 	by   Lemma \ref{vtL1},  there exists  $f_0\in V$ such that 
	$\|S(f_0,\bX)\|_\infty =1$ and $|I(f_0)|\ge   D-m$, where 
	$$
	I(f_0):=\Bl\{i\in  \NN:\   \ 1\leq i\leq M,\   |f_0(\bx^i)|\ge 1/\sqrt{2}\Br\}.
	$$ 
	Using \eqref{vt4} and \eqref{vt5}, we obtain 
	\[ \|f_0\|_p^p \leq 2 \|S(f,\bX)\|_p^p =  \f 2 M \sum_{j=1}^M |f_0(\bx^j)|^p\leq  2,\]
	and 
	\begin{align*}
	\|f_0\|_2^2  \ge \f 23 \|S(f,\bX)\|_2^2  \ge  \f 23  \cdot \f 1M \sum_{j\in I(f_0) } |f_0(\bx^j)|^2\ge\f 13 \f{|I(f_0)|}{M}\ge \f{D-m }{3M}.
	\end{align*}
	
Finally, setting $g=2^{-1/p} f_0$, we have  $g\in X_D^p\cap \cZ(\xi)$ and 
	\[ \|g\|_2^2 \ge \frac{2}{3}  \|S(g, \bX)\|_2^2\ge  2^{-2/p} \frac{2}{3}\f{D-m }  {3M},\]
	which completes the proof.
	
\end{proof}

A combination of Theorem \ref{vtT1} and Theorem \ref{vtT2} gives the following result.

\begin{Theorem}\label{vtT3} Let $2\le p<\infty$. Assume that $X_D$ satisfies condition {\bf D1}$(M,p)$. Let	$X_D^p:=\{ f\in X_D:\  \ \|f\|_p\leq 1\}$. For a dictionary $\D_N$ and a positive constant $C_1$ suppose that $v$ is such that $m(\cX_v(\D_N),C_1)<D$.  
 Then we have 
\begin{equation}\label{vt7}	
  \frac{1}{3}\left(\frac{D-m(\cX_v(\D_N),C_1)}{M}\right)^{1/2} \le (2C_1^{-1} +1) \sigma_v(X_D^p,\D_N)_\infty.
\ee
\end{Theorem}

Theorem \ref{vtT3} is a conditional result. In order to apply it we need to check that the subspace $X_D$ of interest satisfies the discretization condition {\bf D1}$(M,p)$. Also, for a given dictionary
$\D_N$ we need to estimate the universal discretization characteristic $m(\cX_v(\D_N),C_1)$. 
We now illustrate on a few examples how Theorem \ref{vtT3} works. We assume that the subspace 
$X_D$ satisfies the Nikol'skii inequality assumption
\be\label{NI}
\|f\|_\infty \le (K_1D)^{1/p} \|f\|_p,\quad \forall f \in X_D.
\ee

{\bf The case $p=2$.} It is proved in \cite{LimT} that any subspace $X_D$ satisfying (\ref{NI}) with $p=2$ has {\bf D1}$(M,2)$ property with some $M \le C(K_1)D$. Let us consider a dictionary 
$\D_N = \Phi_N$ satisfying (\ref{I9}) and (\ref{Riesz}). Then by Theorem 
\ref{BT1} we get
$$
m(\cX_v(\Phi_N),1/2) \le C(K) v \log N\cdot  \log^2(2v)(\log (2v)+\log\log N).
$$
Let $v(D,N)\in\N$ be the largest $v$ satisfying
$$
C(p,K) v \log N\cdot  \log^2(2v)(\log (2v)+\log\log N) \le D/2.
$$
Then Theorem \ref{vtT3} guarantees that 
$$
\sigma_{v(D,N)}(X_D^2,\D_N)_\infty \ge c_0
$$
with a positive absolute constant $c_0$. In the case $\log N$ is of the same order as $\log D$ 
the number $v(D,N)$ is of the order $D(\log D)^{-4}$. 

The reader can compare this result with a powerful result by B.S. Kashin formulated below (see Lemma \ref{2-6-ch2}).

{\bf The case $2<p<\infty$.} It is proved in \cite[Theorem A]{Kos} and \cite[Theorem 2.2]{DT2} that any subspace $X_D$ satisfying (\ref{NI}) with $2<p<\infty$ has {\bf D1}$(M,p)$ property with some $M \le C(K_1,p) D(\log  D)^{\al_p}$, $\al_p :=\min\{p, 3\}$. Let us consider a dictionary 
$\D_N = \Phi_N$ satisfying (\ref{I9}) and (\ref{Riesz}).  
Let $v(D,N)\in\N$ be as above.
 
Then Theorem \ref{vtT3} guarantees that 
$$
\sigma_{v(D,N)}(X_D^p,\D_N)_\infty \ge \f {c_1} {(\log  D)^{\al_p/2}}
$$
with a positive constant $c_1 = c_1(K_1,p)$. In the case $\log N$ is of the same order as $\log D$ 
the number $v(D,N)$ is of the order $D(\log D)^{-4}$.

{\bf \ref{vt}.2. Connection to numerical integration.} We now use Theorem \ref{vtT1} to prove an inequality between $\sigma_v(\bF,\Phi_N)_\infty$ and 
best error of numerical integration for the class $\bF$. We introduce some notation. Let $\Omega$ be a compact subset of $\R^d$ and $\mu$ be a probability measure on $\Omega$. Denote
$$
I_\mu(f) := \int_\Omega fd\mu.
$$
For a function class $\bF\subset \C(\Omega)$ consider the best error of numerical integration by cubature formulas with $m$ knots:
$$
\kappa_m(\bF) := \inf_{\xi^1,\dots,\xi^m;\lambda_1,\dots,\lambda_m} \sup_{f\in \bF} \left|I_\mu(f) - \sum_{j=1}^m \ld_\nu f(\xi^\nu)\right|,
$$
where $\inf$ is taken over all sets $\{\xi^1,\dots,\xi^m\}\subset \Omega$ and all sets of numbers $\{\lambda_1,\dots,\lambda_m\}$. 
Clearly,
\be \label{vt9} \k_m (\bF) \ge \inf_{\xi\in\Og^m} \sup_{f\in \bF\cap\cZ(\xi)} |I_\mu(f)|.\ee
In particular, \eqref{vt6} gives a lower bound of $\sqrt{\k_m (\bF)}$ for the class of functions $\bF:=\{ |f|^2:\  \ f\in X_N,\  \ \|f\|_p\leq 1\}$.

We now consider special classes $\bF$, which we denote by $W$, and for which we actually have equality in \eqref{vt9}. 
 Let $V\subset \cC(\Omega)$ be a Banach space and $W:=\{f:\|f\|_V\le 1\}$ be its unit ball. We begin with a simple lemma.

\begin{Lemma}\label{vtL2} For the function class $W$ given above, and any integer $m\ge 1$,   we have
\be\label{vt10}
\inf_{\xi\in\Og^m} \sup_{f\in W\cap \cZ(\xi) } |I_\mu(f)| = \kappa_m(W).
\ee
\end{Lemma}
\begin{proof} By \eqref{vt9}, it is enough to show that for any $\xi\in\Og^m$, 
\[\sup_{f\in W\cap \cZ(\xi) } |I_\mu(f)| \ge  \kappa_m(W).\]
Indeed, 	by  the Nikol'skii duality theorem (see, for instance, \cite{VTbookMA}, p.509) we find
\be\label{vt11}
  \sup_{f\in W\cap \cZ(\xi)} |I_\mu(f)| = \inf_{\lambda_1,\dots,\lambda_m}   \left\|I_\mu - \sum_{\nu=1}^m \lambda_\nu \delta_{\xi^\nu}\right\|_{V'} 
\ee
where $V'$ is a dual (conjugate) to the  Banach space $V$  and $\delta_{\xi^\nu}$ are the Dirac delta functions. 
For each $\lambda_1,\dots,\lambda_m$ we have 
\be\label{vt12}
 \left\|I_\mu - \sum_{\nu=1}^m \lambda_\nu \delta_{\xi^\nu}\right\|_{V'} = \sup_{f\in W} \left|I_\mu(f) - \sum_{\nu=1}^m \lambda_\nu f(\xi^\nu)\right|.
\ee
It follows from the definition of $\kappa_m(W)$   that
for any $\xi$  and each $\lambda_1,\dots,\lambda_m$
\be\label{vt13}
\sup_{f\in W} \left|I_\mu(f) - \sum_{\nu=1}^m \lambda_\nu f(\xi^\nu)\right| \ge \kappa_m(W).
\ee
Combining (\ref{vt11})-(\ref{vt13}) we complete the proof.
\end{proof}

 \begin{Theorem}\label{vtT4} Let $m$, $v$, $N$ be given natural numbers such that $v\le N$.  Assume that a dictionary $\D_N$ is such that there exists a set $\xi:= \{\xi^j\}_{j=1}^m \subset \Omega $, which provides {\it one-sided universal discretization} with constant $C_1$ for the collection $\cX_v(\D_N)$. Suppose that $W$ is the unit ball of a Banach space. Then  we have
\be\label{vt14}
  \kappa_m(W) \le (2C_1^{-1} +1) \sigma_v(W,\D_N)_\infty.
 \ee
 \end{Theorem}
\begin{proof} By Theorem \ref{vtT1} we obtain 
\be\label{vt15}
   (2C_1^{-1} +1) \sigma_v(W,\D_N)_\infty \ge \sup_{f\in W\cap \cZ(\xi)} \|f\|_{L_2(\Omega,\mu)}.
 \ee
   Further
 $$
  \sup_{f\in W\cap \cZ(\xi)} \|f\|_{L_2(\Omega,\mu)} \ge  \sup_{f\in W\cap\cZ(\xi)} \|f\|_{L_1(\Omega,\mu)} \ge  \sup_{f\in W\cap \cZ(\xi)} |I_\mu(f)|.
  $$
  It remains to apply Lemma \ref{vtL2}.

\end{proof}

 As a corollary of Theorems \ref{vtT4} and \ref{BT1} we obtain the following result.
 
  \begin{Theorem}\label{vtT5} Let as above $W$ be the unit ball of a Banach space. 
  	Given a number $K>0$, there exists a constant $C_K>0$ such that for any   compact subset $\Omega$  of $\R^d$, any probability measure $\mu$ on it,   for any integer $1\leq v\leq N$, and for 
  	any dictionary $\Phi_N$  satisfying both  \eqref{I9} and \eqref{Riesz} for the constant $K$,  	
  	we have
\be\label{vt16}
\kappa_m(W) \le 5 \sigma_v(W,\Phi_N)_\infty,
 \ee
 provided
  $$m\ge C_Kv(\log N) (\log(2v))^2 (\log (2v)+\log\log N). $$

\end{Theorem}

{\bf Comment \ref{vt}.2.} It might be interesting to compare Theorems \ref{vtT5} and \ref{vtT4} with some known results. The following inequality was proved in \cite{No} (see also \cite{NoLN})
\be\label{vt17}
\kappa_m(\bF) \le 2d_m(\bF,L_\infty).
\ee
The following Theorem \ref{vtT6} was proved in \cite{VT191}. We recall some definitions. 

{\bf Property A.} For any $f\in \bF$ we have $ (f+1)/2 \in \bF$ and $  (f-1)/2 \in \bF$.

For a compact subset $\Theta$ of a Banach space $X$ we define the entropy numbers as follows
$$
\e_n(\Theta,X) := \inf\{\e: \exists f_1,\dots,f_{2^n}\in \Theta: \Theta\subset \cup_{j=1}^{2^n} (f_j+\e U(X))\}
$$
where $U(X)$ is the unit ball of a Banach space $X$.

\begin{Theorem}\label{vtT6} Assume that a class of real functions   $\bF\subset \cC(\Omega)$ has Property A and is such that for all $f\in \bF$ we have $\|f\|_\infty \le M$ with some constant $M$. Also assume that the entropy numbers of $\bF$ in the uniform norm $L_\infty$ satisfy the condition
$$
  \e_n(\bF,L_\infty) \le n^{-r} (\log (n+1))^b,\qquad r\in (0,1/2),\quad b\ge 0.
$$
Then
$$
\kappa_m(\bF)      \le C(M,r,b)m^{-r} (\log (m+1))^b.
$$
\end{Theorem}

 The following known result gives a connection between entropy numbers and best sparse approximations (see, for instance, \cite{VTbookMA}, p.331, Theorem 7.4.3). 
 
 \begin{Theorem}\label{vtT7} Let a compact subset $\bF$ of a Banach space  $X$ be such that for a dictionary $\D_N$, $|\D_N|=N$, and  a number $r>0$ we have
$$
  \sigma_m(\bF,\D)_X \le m^{-r},\quad m\le N.
$$
Then for $n\le N$
\begin{equation}\label{vt18}
\e_n(\bF,X) \le C(r) \left(\frac{\log(2N/n)}{n}\right)^r.
\end{equation}
\end{Theorem}

\section{Recovery of periodic functions with mixed smoothness}
\label{C}

We begin with some notation.  Given $a>0$, we use  the notation $[a]$ to   denote the largest integer $\leq a$.
Let $\mathbf s=(s_1,\dots,s_d )$ be a  vector  whose  coordinates  are
nonnegative integers
$$
\rho(\mathbf s) := \bigl\{ \mathbf k\in\mathbb Z^d:[ 2^{s_j-1}] \le
|k_j| < 2^{s_j},\qquad j=1,\dots,d \bigr\},
$$
$$
Q_n :=   \cup_{\|\mathbf s\|_1\le n}
\rho(\mathbf s) \quad\text{--}\quad\text{a step hyperbolic cross},
$$
$$
\Gamma(N) := \bigl\{ \mathbf k\in\mathbb Z^d :\prod_{j=1}^d
\max\bigl( |k_j|,1\bigr) \le N\bigr\}\quad\text{--}\quad\text{a hyperbolic cross}.
$$
 For $f\in L_1 (\T^d)$
$$
\delta_{\mathbf s} (f,\mathbf x) :=\sum_{\mathbf k\in\rho(\mathbf s)}
\hat f(\mathbf k)e^{i(\mathbf k,\mathbf x)},\quad \hat f(\mathbf k) := (2\pi)^{-d}\int_{\T^d} f(\bx)e^{-i(\mathbf k,\mathbf x)}d\bx.
$$
Let $Q$ be a finite set of points in $\mathbb Z^d$, we denote
$$
\Tr(Q) :=\left\{ t : t(\mathbf x) =\sum_{\mathbf k\in Q}c_{\mathbf k}
e^{i(\mathbf k,\mathbf x)}\right\} .
$$

Along with the $L_p$ norm
$$
\|f\|_p := \left((2\pi)^{-d}\int_{\T^d} |f(\bx)|^pd\bx\right)^{1/p},\quad 1\le p<\infty,
$$
we consider the Wiener norm (the $A$-norm of $f$, which is the $\ell_1$-norm of the Fourier coefficients of $f$)
$$
\|f\|_A := \sum_{\mathbf k \in \Z^d}|\hat f({\mathbf k })|.
$$
 
In this section we study recovery problems for classes of functions with mixed smoothness. We define these classes momentarily.  
We begin with the case of univariate periodic functions. Let for $r>0$ 
\be\label{C1}
F_r(x):= 1+2\sum_{k=1}^\infty k^{-r}\cos (kx-r\pi/2) 
\ee
and
$$
W^r_p := \{f:f=\varphi \ast F_r,\quad \|\varphi\|_p \le 1\},
$$
where
$$
 \quad (\varphi \ast F_r)(x):= (2\pi)^{-1}\int_\T \ff(y)F_r(x-y)dy.  
$$
It is well known that for $r>1/p$ the class $W^r_p$ is embedded into the space of continuous functions $\cC(\T)$.  

In the multivariate case for $\bx=(x_1,\dots,x_d)$ denote
$$
F_r(\bx) := \prod_{j=1}^d F_r(x_j)
$$
and
$$
\bW^r_p := \{f:f=\varphi\ast F_r,\quad \|\varphi\|_p \le 1\},
$$
where
$$ (\varphi \ast F_r)(\bx):= (2\pi)^{-d}\int_{\T^d} \ff(\by)F_r(\bx-\by)d\by.
$$
 
The above defined classes $\bW^r_p$ are classical classes of functions with {\it dominated mixed derivative} (Sobolev-type classes of functions with mixed smoothness). The reader can find results on approximation properties of these classes in the books \cite{VTbookMA} and \cite{DTU}.

The following classes, which are convenient in studying sparse approximation with respect to the trigonometric system, 
where introduced and studied in \cite{VT150}. Define for $f\in L_1(\T^d)$
$$
f_j:=\sum_{\|\bs\|_1=j}\delta_\bs(f), \quad j\in \N_0,\quad \N_0:=\N\cup \{0\}.
$$
For parameters $ a\in \R_+$, $ b\in \R$ define the class
$$
\bW^{a,b}_A:=\{f: \|f_j\|_A \le 2^{-aj}(\bar j)^{(d-1)b},\quad \bar j:=\max(j,1), \quad j\in \N\}.
$$
We now explain how classes $\bW^r_p$ and $\bW^{a,b}_A$ are related. It is well known (see, for instance, \cite{Tmon}, Ch.2, Theorem 2.1) that for $f\in \bW^r_p$, $1<p<\infty$, one has  
\be\label{C2}
\|f_j\|_p \le C(d,p,r)2^{-jr},\quad j\in \N.
\ee
The known results (see \cite{Tmon}, Ch.1, Theorem 2.2) imply for $1<p\le 2$
\be\label{C3}
\|f_j\|_A \le C(d,p)2^{j/p} j^{(d-1)(1-1/p)} \|f_j\|_p  
\ee
\be\label{C3'}
\le C(d,p,r) 2^{-(r-1/p)j}j^{(d-1)(1-1/p)}    .
\ee
Therefore, class $\bW^r_p$ is embedded into the class $\bW^{a,b}_A$ with $a=r-1/p$ and $b = 1-1/p$. 

The following bound for sparse approximation was proved in \cite{VT150} (see Lemma 2.1 and its proof there). 

\begin{Lemma}\label{L2.1}   There exist two constants $c(a,d)$ and $C(a,b,d)$ such that for any $v\in\N$ there is a constructive method $A_v$ based on greedy algorithms, which provides a $v$-term approximant 
from $\Tr(\Gamma(M))$, $|\Gamma(M)|\le v^{c(a,d)}$, with 
the bound for $f\in \bW^{a,b}_A$
$$
\|f-A_v(f)\|_\infty \le C(a,b,d)  v^{-a-1/2} (\log v)^{(d-1)(a+b)+1/2}.      
$$

\end{Lemma}

Lemma \ref{L2.1} and Theorem \ref{IT4} imply the following bound for the error of nonlinear $v$-sparse least squares recovery 
in $L_2$ in the classes $\bW^{a,b}_A$. Let $v\in\N$ be given and let $M$ be from Lemma \ref{L2.1}. Define the system $\Phi_N$ as  the orthonormal 
basis $\{e^{i(\bk,\bx)}\}_{\bk\in \Gamma(M)}$ of the $\Tr(\Gamma(M))$.

\begin{Theorem}\label{CT1}   There exist two constants $c'(a,d)$ and $C'(a,b,d)$ such that       for any $v\in\N$ we have the bound (with $\Phi_N$ defined above)
\begin{equation}\label{C4}
 \varrho_{m,v}^{ls}(\bW^{a,b}_A,\Phi_N,L_2(\T^d)) \le C'(a,b,d)  v^{-a-1/2} (\log v)^{(d-1)(a+b)+1/2}      
\end{equation}
for any $m$ satisfying 
$$
m\ge c'(a,d) v(\log(2v))^4.
$$
\end{Theorem}
\begin{proof} Let $v\in\N$ be given and let $M$ be from Lemma \ref{L2.1}. Consider the orthonormal 
basis $\{e^{i(\bk,\bx)}\}_{\bk\in \Gamma(M)}$ of the $\Tr(\Gamma(M))$ to be the system $\Phi_N$.
Then $N=|\Gamma(M)|$ and by Lemma \ref{L2.1} we obtain $N\le v^{c(a,d)}$. We now apply Theorem \ref{IT4} with $\Omega =\T^d$, $\mu$ -- the normalized Lebesgue measure on $\T^d$, and $\bF =\bW^{a,b}_A$.  This completes the proof. 
\end{proof}
 
The following classes were introduced and studied in \cite{VT152} (see also \cite{VTbookMA}, p.364)
$$
\bW^{a,b}_p:=\{f: \|f_j\|_p \le 2^{-aj}(\bar j)^{(d-1)b},\quad \bar j:=\max(j,1),\quad j\in \N_0\}.
$$
It is well known that the class $\bW^r_p$ is embedded in the class $\bW^{r,0}_p$ for $1<p<\infty$ (see (\ref{C2}) above). 
Classes $\bW^{a,b}_p$ provide control of smoothness at two scales: $a$ controls the power type smoothness and $b$ controls the logarithmic scale smoothness. Similar classes with the power and logarithmic scales of smoothness are studied in the book of Triebel \cite{Tr}.

The inequality (\ref{C3}) implies that for $1< p \le 2$ the class $\bW^{a,b}_p$ is embedded into 
$\bW^{a',b'}_A$ with $a'= a-1/p$, $b'= b+1-1/p$. Therefore, Theorem \ref{CT1} implies the 
following bounds for the classes $\bW^{a,b}_p$. In further formulations we use the term 
{\it constant} for a positive number. 

\begin{Theorem}\label{CT2}  Let $1<p\le 2$.  There exist two constants $c=c(a,d,p)$ and $C=C(a,b,d,p)$ such that       for any $v\in\N$ we have the bound for $a>1/p$ (with $\Phi_N$ defined above)
 \begin{equation}\label{C5}
  \varrho_{m,v}^{ls}(\bW^{a,b}_p,\Phi_N,L_2(\T^d)) \le C   v^{-a+1/p-1/2} (\log v)^{(d-1)(a+b+1-2/p)+1/2}      
\end{equation}
for any $m$ satisfying 
$$
m\ge c  v(\log(2v))^4.
$$
\end{Theorem}

In particular, we have the following result for the $\bW^{r}_p$ classes. 

\begin{Theorem}\label{CT3}  Let $1<p\le 2$.  There exist two constants $c=c(r,d,p)$ and $C=C(r,d,p)$ such that       for any $v\in\N$ we have the bound for $r>1/p$ (with $\Phi_N$ defined above)
\begin{equation}\label{C6}
 \varrho_{m,v}^{ls}(\bW^{r}_p,\Phi_N,L_2(\T^d)) \le C   v^{-r+1/p-1/2} (\log v)^{(d-1)(r+1-2/p)+1/2}      
\end{equation}
for any $m$ satisfying 
$$
m\ge c  v(\log(2v))^4.
$$
\end{Theorem}

We now discuss a very interesting effect, which was noticed in \cite{JUV}. We compare 
two recovery characteristics -- the linear one $\varrho_m(\bW^{a,b}_p,L_2(\T^d))$ and the nonlinear one $\varrho_{m,v}^{ls}(\bW^{a,b}_p,\Phi_N,L_2(\T^d))$. On the one hand, by (\ref{I2}) we have
\be\label{C7}
\varrho_m(\bW^{a,b}_p,L_2(\T^d)) \ge d_m(\bW^{a,b}_p,L_2(\T^d)).
\ee
For brevity, in our further discussion we use signs $\ll$ and $\asymp$. Relation $a_m\ll b_m$ means 
$a_m\le Cb_m$ with $C$ independent of $m$. Relation $a_m\asymp b_m$ means $a_m\ll b_m$ and $b_m\ll a_m$.
It is well known (see, for instance, \cite{VTbookMA}, p. 216) that for $1<p\le 2$
$$
d_m(\bW^{r}_p,L_2(\T^d)) \asymp \left((\log m)^{d-1}/m\right)^{r-\eta},\quad \eta:= 1/p-1/2, \quad r>\eta.
$$
In the same way one can obtain the following relation for $1<p\le 2$
\be\label{C8}
d_m(\bW^{a,b}_p,L_2(\T^d)) \asymp \left((\log m)^{d-1}/m\right)^{a-\eta}(\log m)^{(d-1)b}, \quad a>\eta.
\ee
On the other hand, Theorem \ref{CT2} implies that for $1<p\le 2$
\be\label{C9} 
\varrho_{m,v}^{ls}(\bW^{a,b}_p,\Phi_N,L_2(\T^d)) \ll \left((\log m)^4/m\right)^{a-\eta} (\log m)^{(d-1)(a+b-2\eta)+1/2}  .
\ee
It is clear that in the case $1<p<2$ for large enough $d$ the right hand side of (\ref{C9}) goes to zero faster than 
the right hand side of (\ref{C8}). This follows from the fact that the exponents for $m$ in both (\ref{C8}) and (\ref{C9}) coincide and the exponent of $\log m$ in (\ref{C9}) ($(d-1)(a+b-2\eta)+1/2+4(a-\eta)$) is smaller than the exponent of $\log m$ in (\ref{C8})) ($(d-1)(a+b-\eta)$) for large $d$. Therefore, for such $d$ we have for $1<p<2$
\be\label{C10}
\varrho_{m,v}^{ls}(\bW^{a,b}_p,\Phi_N,L_2(\T^d)) = o\left(\varrho_m(\bW^{a,b}_p,L_2(\T^d))\right),
\ee
which demonstrates advantages of nonlinear methods over linear ones. 

{\bf Comment \ref{C}.1.} It is clear from the above argument that the relation (\ref{C10}) is a corollary 
of the corresponding relation between the Kolmogorov $n$-width and the best  $n$-term approximation with respect to the trigonometric system. This phenomenon is known. Actually, the following relation holds for $1<q<p\le 2$, $r>2(1/q-1/p)$ (see, for instance, \cite{Tmon}, Chapter IV, Section 2)
$$
 \sigma_n(\bW^r_q,\Tr^d)_p \asymp d_n(\bW^r_q,L_p) (\log n)^{(d-1)(1/p-1/q)},
 $$
 where $\Tr^d := \{e^{i(\bk,\bx)}\}_{\bk\in \Z^d}$ is the trigonometric system defined on $\T^d$. 

{\bf Comment \ref{C}.2.} In this section we demonstrated applications of Theorem \ref{IT4} 
with the system $\Phi_N$ being a subsystem of the trigonometric system. Namely, we used 
inequality (\ref{I14}) and known results on best $v$-term approximation in the uniform norm.  
Remark \ref{BR1} shows that one can obtain better results on best $v$-term 
approximation in the norm $\|\cdot\|_{L_2(\T^d,\mu_\xi)}$, $\xi \in (\T^d)^m$, than in the uniform norm. Inspecting the proof of Lemma \ref{L2.1} in \cite{VT150}, we 
find that it gives the following version of it in the case of norm $\|\cdot\|_{L_2(\T^d,\mu_\xi)}$.

\begin{Lemma}\label{L2.1a}   There exist two constants $c(a,d)$ and $C(a,b,d)$ such that for any $\xi\in (\T^d)^m$ and $v\in\N$ there is a constructive method $A_{v,\xi}$ based on greedy algorithms, which provides a $v$-term approximant 
from $\Tr(\Gamma(M))$, $|\Gamma(M)|\le v^{c(a,d)}$, with 
the bound for $f\in \bW^{a,b}_A$
$$
\|f-A_{v,\xi}(f)\|_{L_2(\T^d,\mu_\xi)} \le C(a,b,d)  v^{-a-1/2} (\log v)^{(d-1)(a+b)}.      
$$

\end{Lemma}

Lemma \ref{L2.1a} and Theorem \ref{IT4} imply the following analog of Theorem \ref{CT1}.   Let $v\in\N$ be given and let $M$ be from Lemma \ref{L2.1a}. Define the system $\Phi_N$ as  the orthonormal 
basis $\{e^{i(\bk,\bx)}\}_{\bk\in \Gamma(M)}$ of the $\Tr(\Gamma(M))$.

\begin{Theorem}\label{CT1a}   There exist two constants $c'(a,d)$ and $C'(a,b,d)$ such that       for any $v\in\N$ we have the bound (with $\Phi_N$ defined above)
\begin{equation}\label{C4}
 \varrho_{m,v}^{ls}(\bW^{a,b}_A,\Phi_N,L_2(\T^d)) \le C'(a,b,d)  v^{-a-1/2} (\log v)^{(d-1)(a+b)}      
\end{equation}
for any $m$ satisfying 
$$
m\ge c'(a,d) v(\log(2v))^4.
$$
\end{Theorem}

Theorem \ref{CT1a} improves Theorem \ref{CT1} because it has the smaller exponent $a+b$ of the 
$\log v$ factor   instead of $a+b+1/2$. 
In the same way we can improve Theorems \ref{CT2} -- \ref{CT3}, which are corollaries of Theorem \ref{CT1}. Namely, we delete $1/2$ from the exponent of $\log v$ in the bounds.

	\section{Sampling recovery by  Gegenbauer   polynomials}
	\label{G}

	In  this section, we  investigate sampling recovery in the univariate nonperiodic case, aiming  to derive several new results on the best $v$-term approximation with respect to  the system of Gegenbauer polynomials,  which  extend and complement the corresponding results 
	from \cite{JUV}.

We start with some necessary notations. 
For $\al>-\f12$, denote by  $\mu_\al$  the probability measure on $[-1,1]$ given by
$d\mu_\alpha(x)= v_\al(x) d x$, where  $v_\alpha(x):=c_\alpha\left(1-x^2\right)^\alpha$ and  $c_\alpha=\left(\int_{-1}^1\left(1-x^2\right)^\alpha d x\right)^{-1}$.
Let  $\mathcal{B}_\alpha=\{L_n^\al\}_{n=0}^\infty$ denote the  system of  orthonormal  polynomials with respect to the inner product $\langle\cdot,\cdot\rangle_{L^2\left(\mu_\alpha\right)}$   of $L_2(\mu_\al)$
 given by
$$
\langle f, g\rangle_{L^2\left(\mu_\alpha\right)}=\int_{-1}^1 f(x) \overline{g(x)} d \mu_\alpha(x),
$$
and
 $L_n^\al$ is
the orthonormalized   Gegenbauer  polynomial of degree $n$   with respect to the weight $v_\al$ on  $[-1,1]$ . Here and  throughout this section,   we use the notation $L_p(\mu_\al)$ to denote the Lebesgue space 
$L_p ([-1,1], \mu_\al)$  for the sake of convenience.
In the case of $\al=0$,   $\CB_0:=\{L_n^0\}_{n=0}^\infty$ is the system of orthonormalized Legendre polynomials on $[-1,1]$. 	
   Let $\mathcal{P}^n$ denote the space of all real algebraic polynomials in one variable of degree at most $n$.  Define   $\CB_\al^N:=\CB_\al\cap\mathcal{P}^N.$  Then $\CB_\al^N$ is an orthonormal basis of the space $\mathcal{P}^N$ equipped with the norm of $L_2([-1,1],\mu_\al)$.

One of the main difficulties in the study of  sparse approximation by  Gegenbauer   polynomials comes from the fact that the system $\CB_\al$ is not uniformly bounded on $[-1,1]$  for any $\al>-\f12$, which is different from the usual trigonometric  system. Indeed,
according to (\cite[ p. 81, (4.7.1); p.82, (4.7.15); p.169, (7.32.5)]{Sz}, we have  \begin{equation}\label{1-1-eq-ch2} |L_n^\al (x)|\leq C(\al)
\min\Bl\{ (n+1)^{\al+\f12}, (1-x^2)^{-\f\al2-\f14}\Br\},\end{equation}
and
\begin{equation}
\max_{x\in [-1,1]} |L_n^\al(x)|=L_n^\al(1) \sim (n+1)^{\al+\f12}.\end{equation}
Here and throughout this section, we use the notation $a_n\sim b_n$ for $n=1,2,\ldots$ to mean that there exist two positive constants $C_1, C_2$ independent of $n$ such that $C_1 a_n\leq b_n\leq C_2 a_n$ for all $n=1,2,\ldots$.

On the other hand,  setting $\vi_n^\al(x) =C(\al)^{-1} L_n^\al(x) (1-x^2)^{\f \al2+\f14}$ for  $n\in\NN_0$,    and using \eqref{1-1-eq-ch2}, we know  that $\{\vi_n^\al\}_{n=0}^\infty$ is a uniformly bounded  orthogonal system with respect to the probability measure $\f1\pi (1-x^2)^{-\f12}\, dx$ on $[-1,1]$.
Thus, Theorem \ref{BT1},   our result on universal discretization, is applicable   to the system $\{\vi_n^\al\}_{n=0}^\infty$. This immediately leads to the following result:

%

\begin{Theorem} \label{thm-6-1}
 For   $\al>-\f12$, and any integer $1\leq v\leq N$, there exist points  $\xi_1,\ldots, \xi_m\in [-1,1]$ with
	$$
	m \leq   C(\al)  v \log N\cdot \log^2(2v) \cdot (\log (2v)+\log\log N),
	$$
	such that
	\begin{equation*}
	\f12\|f\|_{L_2(\mu_\al)}^2 \le \f 1m \sum_{j=1}^m w(\xi_j)
	|f(\xi_j)|^2 \le \f32\|f\|_{L_2(\mu_\al)}^2,\   \   \ \forall f\in  \Sigma_v(\CB_\al^N),
	\end{equation*}
	where $w(x)=\pi^{-1}\sqrt{1-x^2}$ for $x\in [-1,1]$.
\end{Theorem}

Given ${\xi}=(\xi_1,\ldots,\xi_m)\in [-1,1]^m$, we  define
\[ \bw_m({\xi}) =\f 1m (w(\xi_1), w(\xi_2), \ldots, w(\xi_m)),\]
 where  $w(x)=\pi^{-1}\sqrt{1-x^2}$ for  $x\in [-1,1]$.

For  an $N$-dimensional space $X_N\subset \cC\big([-1,1]\big)$ and $\xi=(\xi_1,\ldots,\xi_m)\in [-1,1]^m$,
we recall that the weighted least squares recovery operator is defined as in the introduction  by 
$$
\ell 2\bw_m(\xi,X_N)(f):=\text{arg}\min_{u\in X_N} \|S(f-u,\xi)\|_{2,\bw_m(\xi)},
$$
where
$$
\|S(g,\xi)\|_{2,\bw_m(\xi)}:= \left(\f 1m \sum_{\nu=1}^m w(\xi_j)  |g(\xi_j)|^2\right)^{1/2} .
$$

Given an integer $1\leq v\leq N$,
we consider the following sparse recovery characteristic
$$
\varrho^{wls}_{m,v}(\mathbf{F},\CB_\al^N,L_2) := \inf_{\xi =\{\xi_1,\dots,\xi_m\}} \sup_{f\in \mathbf{F}} \min_{L\in \cX_v(\CB_\al^N)} \|f-	\ell 2\bw_m(\xi,L)(f)\|_{L_2(\mu_\al)}
$$
where we recall that $\cX_v(\CB_\al^N)$ denotes the collection of all  linear spaces spanned by $ L_j^\al:\  \ j\in J$
with $J\subset [1,N]\cap \mathbb{N}$ and $|J|=v$.

\begin{Theorem}\label{thm-8-2} For  $\al>-\f12$, any integer $1\leq v\leq N$,  and 	any compact subset $\mathbf{F}$ of $\cC[-1,1]$, we have
	\begin{equation}
	\varrho_{m,v}^{wls}(\mathbf{F},\CB_\al^N,L_2(\mu_\al) ) \le 5 \sigma_v(\mathbf{F},\CB_\al^N)_\infty,
	\end{equation}
	provided
		$$
	m \ge   C(\al)  v \log N\cdot \log^2(2v) \cdot (\log (2v)+\log\log N).
	$$

\end{Theorem}

For the remainder of this section, we restrict our attention to  some    special function classes where useful estimates on the quantities $\sigma_v(\mathbf{F},\CB_\al^N)_\infty$ can be obtained.

The following nonperiodic  Wiener-type class of functions   have been considered  by  Rauhut and Ward \cite[Section 7]{RW}  for the special case of $\al=0$ (i.e., the case of Legendre polynomials),
and by  Jahn, Ullrich, and   Voigtlaender \cite[Definition 4.7]{JUV} for both  $\al=0$ and $\al=-\f12$ (i.e., the case of  Chebyshev polynomials).
We refer to \cite{RW, JUV} for more background information on these spaces.

\begin{Definition}
	
	For $\alpha >-\f12$, $r>\frac{1}{2}+\alpha$ and $0<\ta \leq 1$, we  define
	\begin{align*}
	\mathcal{A}_{\alpha, \ta}^r:=\Bigl\{f\in C[-1,1]:\   \  \|f\|_{	\mathcal{A}_{\alpha, \ta}^r}^\ta:=
	\sum_{n \in \mathbb{N}_0}\left|(1+n)^r\left\langle f, L^\al_n\right\rangle_{L^2\left(\mu_\alpha\right)}\right|^\ta\leq 1\Br\}.
	\end{align*}

\end{Definition}
 The condition $r>\al+\f12$ in the above definition ensures the corresponding space can be continuously embedded into  the space $\cC[-1,1]$.

Our next aim is to  determine the sharp asymptotic order of the quantities $\sigma_n\left(\mathcal{A}_{\al,\ta}^r, \mathcal{B}_\al\right)_{\infty}$.
In the case of Legendre polynomials (i.e., $\al=0$), it was  conjectured in  \cite[Remark 4.9]{JUV} that  $\sigma_n\left(\mathcal{A}_{0,1}^r, \mathcal{B}_0\right)_{\infty} \lesssim n^{-r}$  for any  $r>\f12$.
Our main result, Theorem \ref{thm-8-3} below, gives a  bound   that is  better than this conjectured one.  In particular, the sharp  estimates we obtained below can be used to improve  and extend the estimates of Corollary 4.8 of \cite{JUV}.

\begin{Theorem}\label{thm-8-3}
	Let $0<\ta\leq 1$  and   $\al>-\f12$. Then for $r>\al+\f12$, we have
 	$$c_1  n^{-(r+\f 1\ta-\f12)}\leq \sigma_n\left(\mathcal{A}_{\al,\ta}^r, \mathcal{B}_\al\right)_{\infty}\leq C_2 n^{-(r+\f 1\ta-\f12)}
	$$
	for some positive constants $c_1$ and $C_2$ depending only on $\al$, $\ta$ and $r$.
\end{Theorem}

\begin{proof} We start with the proof of the lower estimates, which is simpler due to the following deep result of Kashin \cite[Corollary 2]{Kas}.

\begin{Lemma} \label{2-6-ch2} Let $( \Omega, \mu)$ be a probability space, and let  $\D:=\{ \phi_j\}_{j=1}^\infty$ be an orthonormal basis of the space $L_2(\Og, \mu)$.  Then  there exists a universal constant $c_0\in (0, 1)$ such that  for any  orthonormal system   $\{\psi_j\}_{j=1}^N\subset L_2(\Og, \mu)$, any integer $N\in\NN$   and for the  function class  $$ \bw_N:=\Bl\{ \sum_{ j=1}^N a_j \psi_j:\   \    \max_{ 1\leq j\leq N}
	|a_j|\leq 1\Br\},$$ we have
	\[ \sa_n (\bw_N, \D)_{L_2(\Og, \mu)}\ge \f {\sqrt{3N}} 2,\    \    \   \  \forall n\in [1, c_0N]\cap\NN.\]
	
\end{Lemma}

	Let $m\in \NN$ be such that $2^m \leq n< 2^{m+1}$. Let $m_1=m+m_0$, where  $m_0$ is  the smallest integer such that $2^{m_0-1} c_0>1$.  Define
	\[ W:= \Bl\{  2^{-(m_1+1) (r+\f 1\ta)}\sum_{2^{m_1}\leq j<2^{m_1+1}} a_j L_j^\al:\  \ a_j\in\RR,\   \ |a_j|\leq 1\Br\}.\]
	Since $n<2^{m+1} \leq c_0 2^{m_1}$, using Lemma \ref{2-6-ch2} with $N=2^{m_1}$, we deduce
	\be \sa_n (W, \CB_\al)_{L_2(\mu_\al)}\ge \f {\sqrt{3N} } 2  2^{-(m_1+1) (r+\ta^{-1})}\ge c n^{-(r+\f 1\ta-\f12)}.\label{8-4}\ee
	On the other hand,  however, for any $$f= 2^{-(m_1+1) (r+\f 1\ta)}\sum_{2^{m_1}\leq j<2^{m_1+1}} a_j L_j^\al\in W,$$ we have
	\begin{align*} \|f\|_{\cA_{\al,\ta}^r}&\leq 2^{-(m_1+1) (r+\f 1\ta)}\Bl( \sum_{2^{m_1}\leq j<2^{m_1+1}} (j+1)^{r\ta}\Br)^{1/\ta}\\
	&\leq 2^{-(m_1+1) (r+\f 1\ta)} 2^{(m_1+1) r} 2^{m_1/\ta}\leq 1
	\end{align*}
	implying that $W\subset \cA_{\al,\ta}^r$.
	Thus, using \eqref{8-4}, we obtain
	\begin{align*}
	 \sa_n (\cA_{\al,\ta}^r, \CB_\al)_{\infty}\ge \sa_n (W, \CB_\al)_{L_2(\mu_\al)}\ge c n^{-(r+\f 1\ta-\f12)},
	\end{align*}
	which gives the desired lower bound.

It remains to prove  the upper bound.
We  need
the following result, which follows directly from  \cite[Theorem 1.2]{Dai06}.

\begin{Lemma}\label{2-1-ch2}\textnormal{\cite[Theorem 1.2]{Dai06}}
	Assume that  $\al>-\f12$, $N\in\NN$ and $2\leq p<\infty$. Then for any $f\in \mathcal{P}^N$ and any integer $1\leq n\leq N$,
	\begin{equation}\label{2-1-eq-ch2} \sa_n(f;  \mathcal{B}_\al^{2N} )_{\infty} \leq C(\al)
	\Bigl ( \frac Nn\Bigr )
	^{(2\al+2)/p}
	\Bigl (\log \frac{e N}n\Bigr)^{(2\al+1)/p}\|f\|_{L_p(\mu_\al)}.\end{equation}
\end{Lemma}

Now we turn to the proof of the upper bound.
	Let $m\in \NN$ be such that $2^m \leq n< 2^{m+1}$. Without loss of generality, we may assume that  $m\ge 20$.
	Let $f\in \mathcal{A}_{\al,\ta}^r([-1,1])$ be such that
	\begin{equation}\label{2-1a}
	\sum_{j \in \mathbb{N}_0}(1+j)^{r\ta}\left|\left\langle f, L_j^\al\right\rangle\right|^\ta=1.
	\end{equation}
	Here and elsewhere in  the proof, $\langle\cdot,\cdot\rangle$ denotes the inner product of $L_2(\mu_\al)$.
	Let $\Ld_0\subset\NN_0$ be such that $|\Ld_0| =2^{m-2}$, and
	\[ \min_{j\in\Ld_0}|\langle f, L_j^\al \rangle| (1+j)^r \ge  \sup_{j\in\NN_0\setminus \Ld_0}|\langle f, L_j^\al \rangle| (1+j)^r.\]
	Then \eqref{2-1a} implies that
	\begin{equation}\label{2-2a}
	|\langle f, L_j^\al \rangle|^\ta (1+j)^{r\ta}\leq  2^{-m+2},\    \    \  \forall j\in\NN_0\setminus \Ld_0.
	\end{equation}
	Let $\Ld_1:=\Ld_0\cup \{j\in\NN_0:\  \ j<  2^{m-2}\}$.
	Let
	$$g_1=\sum_{j\in \Ld_1} \langle f, L_j^\al\rangle  L_j^\al\  \ \text{and}\   \ f_1=f-g_1=\sum_{j\in \NN_0\setminus \Ld_1} \langle f, L_j^\al\rangle  L_j^\al.$$
	We write  $f_1=\sum_{k=m-1}^\infty \Delta_k,$
	where
	$ \Delta_k: =\sum_{ 2^{k-1}\leq j<2^k} \langle f_1, L_j^\al\rangle L_j^\al.$

	Let $\k_0$ denote the smallest  integer
	$\ge \f {r+\f 1\ta-\f12}{r-\al-\f12} +1$.
	Using \eqref{1-1-eq-ch2} and  \eqref{2-1a}, we obtain
	\begin{align}
	\Bl\|\sum_{k=\k_0m+1}^\infty \Delta_k \Br\|_\infty& = \Bl\|\sum_{j\ge 2^{\k_0 m}} \langle f_1, L_j^\al\rangle L_j^\al \Br\|_\infty
	\leq  C\sum_{j\ge 2^{k_0 m}}|\langle f_1,L_j^\al\rangle | j^{(\al+\f12)} \notag\\
	&\leq C 2^{-\k_0 m(r-\al-\f12)}\sum_{j\ge 2^{\k_0 m}}|\langle f_1,L_j^\al\rangle | j^{r}\notag\\
	&\leq C 2^{- m(r+\f 1\ta -\f12 )}  \Bl(\sum_{j\ge 2^{\k_0 m}}|\langle f_1,L_j^\al\rangle |^\ta  j^{r\ta}\Br)^{1/\ta}\leq C n^{-(r+\f 1\ta-\f12)}.\label{2-3a}
	\end{align}

	Thus, to complete the proof, it is enough to show there exist   a sequence $\{n_k\}_{k=m-1}^{\k_0 m}$ of positive integers and  a sequence $\{p_k\}_{k=m-1}^{\k_0m} $ of polynomials such that $\sum_{k=m-1}^{\k_0 m} n_k\leq 2^{m-1}$, $p_k\in \Sigma_{n_k} (\CB_\al)$, and
	\begin{equation}\label{2-4a}
	\sum_{k=m-1}^{\k_0 m} \|\Delta_k -p_k\|_\infty \leq C n^{- (r+\f 1\ta -\f12) }.
	\end{equation}
	 Indeed, once this is proven, then  setting  $$g:=g_1+\sum_{k=m-1}^{\k_0 m} p_k,$$
	we have $g\in\Sigma_n(\CB_\al)$ and
	\[ \sa_n(f, \CB_\al)_\infty\leq \|f-g\|_\infty \leq \sum_{k=m-1}^{\k_0 m} \|\Delta_k-p_k\|_\infty +\Bl\|\sum_{k=\k_0m+1}^\infty \Delta_k \Br\|_\infty,\]
	which, combined with \eqref{2-3a} and \eqref{2-4a}, will  yield the desired upper estimate.

	The sequence  $\{n_k\}_{k=m-1}^{\k_0 m}$ of positive integers is defined as follows:
	  $$n_k: = [ (
	k-m+2)^{-2}  \cdot 2^{m-2}],\  \ m-1\leq k \leq \k_0 m.$$  Clearly,  $n_k\leq 2^{k-1}$ for $k\ge m-1$,  and
	\[ \sum_{k=m-1}^{
		\k_0 m } n_k\leq   2^{m-2}\sum_{k=1}^{
		\infty}  \f{1} {k^2}
	\leq 2^{m-1},\]
	 where we recall that   the notation $[a]$  denotes the largest integer $\leq a$ for each given  $a>0$.
	It remains to show the existence of a sequence of  polynomials $p_k\in \Sigma_{n_k} (\CB_\al)$,  $m-1\leq k\leq \k_0 m$ satisfying \eqref{2-4a}.

%

	
	To this end, let $$\max\Bl\{2,  \f {2\al+2}r\Br\}<p<\f {2\al+2} {\al+\f12}$$ be a fixed parameter depending only on $r$ and  $\al$.  By \eqref{1-1-eq-ch2}, we have
	\[\|L_n^\al\|_p \leq C(\al).\]
	Since the Banach space $L_p$ is $2$-smooth for $p>2$, it follows by Theorem~9.2.2 of \cite[p. 455]{VTbookMA} that
	for  each
	integer $ k\in [ m-1, \k_0 m]$,  there exists   a polynomial $ p_{k,1}\in \Sigma _{[n_k/2]}(\CB_\al^{2^{k+1}})$ such
	that
	\begin{align*}
	\|\Delta_k - p_{k,1} \|_p &\leq C \sqrt{p}n_k^{-\f12} \sum_{j= 2^{k-1}}^{2^k} |\langle f_1, L_j^\al\rangle | \|L_j^\al\|_p
	\leq C n_k^{-\f12} 2^{-kr}
	\sum_{j= 2^{k-1}}^{2^k} |\langle f_1, L_j^\al\rangle | j^r \\
	&\leq C 2^{-m(\f 1 \ta -\f12)}  2^{-kr} (k-m+2).
	\end{align*}
	On the other hand, for each
	integer $ k\in [ m-1, \k_0 m]$,  applying  Lemma
	$\ref{2-1-ch2}$ to the polynomial $\Delta_k - p_{k,1}\in \mathcal{P}^{2^{k+1}}$, we obtain   $ p_{k,2}\in \Sigma _{[n_k/2]}(\CB_\al)$ such
	that
	\begin{align*}
	\|\Delta_k - p_{k,1}-p_{k,2} \|_\infty &\leq C   \|\Delta_k-p_{k,1}\|_{p} \Bigl( \f {
		2^k}{ n_k} \Bigr) ^{(2\al+2)/p} \Bigl( \log \f {e 2^k } { n_k} \Bigr) ^{(2\al+1)/p}\\
	&\leq C 2^{-m(\f 1 \ta -\f12)}  2^{-kr}  2^{(k-m+1)(2\al+2)/p}(k-m+2)^{\f {6\al+5}p+1}.
	\end{align*}
	Let $p_k=p_{k,1}+p_{k,2}$, Then $p_k\in\Sigma_{n_k}$ and moreover,
	\begin{align*}
	\sum_{k=m-1}^{\k_0 m} \|\Delta_k - p_{k} \|_\infty&=\sum_{k=m-1}^{\k_0 m} \|\Delta_k - p_{k,1}-p_{k,2} \|_\infty \\
	&\leq C  2^{-m(r+\f 1\ta -\f12)} \sum_{k=1}^\infty  2^{-k \bigl(r -\f{2\al+2} p\bigr)}(k+2)^{\f {6\al+5}p+1}\\
	&\leq C n^{-r-\f 1\ta +\f12}.
	\end{align*}

\end{proof}

Combining  Theorem \ref{thm-8-3} with Theorem \ref{thm-8-2}, we deduce

\begin{Corollary} For  $\al>-\f12$, any integer $1\leq v\leq N$,  $0<\ta\leq 1$ and $r>\al+\f12$, we have
	\begin{align*}
	\varrho_{m,v}^{wls}&(\cA_{\al,\ta}^r,\CB_\al^N,L_2(\mu_\al) )
	 C(\al,\ta, r) 
	n^{-(r+\f 1\ta -\f12)}
	\end{align*}
	provided
	$$
	m \ge   C(\al)  v \log N\cdot \log^2(2v) \cdot (\log (2v)+\log\log N).
	$$

\end{Corollary}

\section{Discussion}
\label{D}

We pointed out in the Introduction that the recent paper \cite{JUV} motivated us to estimate 
 the accuracy of the recovery algorithm  $LS(\xi,\cX_v(\D_N))(\cdot)$ in terms of errors of 
 sparse approximation with respect to special systems $\D_N)$. We begin this section with 
 providing one more argument that motivated us to closely study the least squares algorithms. It is based on the recent result from \cite{VT183} on the linear recovery.  
 Recall the setting 
 of the optimal linear recovery. For a fixed $m$ and a set of points  $\xi:=\{\xi^j\}_{j=1}^m\subset \Omega$, let $\Phi $ be a linear operator from $\bbC^m$ into $L_p(\Omega,\mu)$.
Denote for a class $\bF$ (usually, centrally symmetric and compact subset of $L_p(\Omega,\mu)$)
$$
\varrho_m(\bF,L_p) := \inf_{\text{linear}\, \Phi; \,\xi} \sup_{f\in \bF} \|f-\Phi(f(\xi^1),\dots,f(\xi^m))\|_p.
$$
In other words, for  given  $m$, a system of functions $\psi_1(\bx),
\dots,\psi_m (\bx)$ from $L_p(\Omega,\mu)$, and a set  
$\xi := \{\xi^1,\ldots,\xi^m\}$ from $\Omega$ we define the linear operator
\be\label{I1}
 \Psi_m(f,\xi) :=\sum_{j=1}^{m}f(\xi^j)\psi_j (\bx).
\ee
For a class $\bF$ of functions we define the quantity
$$
\Psi_m (\bF,\xi)_p:=\sup_{f\in \bF}\bigl\|\Psi_m (f,\xi) - f\bigr\|_p.
$$
Then
$$
\varrho_m (\bF, L_p) :=\inf_{\psi_1,\dots,\psi_m;\xi^1,\dots,\xi^m}
\Psi_m (\bF,\xi).
$$
 
 In the above described linear recovery procedure the approximant  comes from a linear subspace of dimension at most $m$. 
It is natural to compare $\varrho_m(\bF,L_p)$  with the 
Kolmogorov widths. Let $X$ be a Banach space and $\bF\subset X$ be a  compact subset of $X$. The quantities  
$$
d_n (\bF, X) :=  \inf_{\{u_i\}_{i=1}^n\subset X}
\sup_{f\in \bF}
\inf_{c_i} \left \| f - \sum_{i=1}^{n}
c_i u_i \right\|_X, \quad n = 1, 2, \dots,
$$
are called the {\it Kolmogorov widths} of $\bF$ in $X$. It is clear that in the case $X=H$ is a Hilbert space the best approximant from a given subspace is provided by the corresponding orthogonal projection. 

We have the following obvious inequality
\be\label{I2}
d_m (\bF, L_p)\le  \varrho_m(\bF,L_p).
\ee

In the paper \cite{VT183} we considered the case $p=2$, i.e. recovery takes place in the Hilbert space $L_2$.
The main result of that paper is the following general inequality.
\begin{Theorem}\label{IT1}  There exist two positive absolute constants $b$ and $B$ such that for any   compact subset $\Omega$  of $\R^d$, any probability measure $\mu$ on it, and any compact subset $\bF$ of $\C(\Omega)$ we have
$$
\ro_{bn}(\bF,L_2(\Omega,\mu)) \le Bd_n(\bF,L_\infty).
$$
\end{Theorem}

Theorem \ref{IT1} was proved in \cite{VT183} with a help of a classical type of algorithms -- weighted least squares algorithms $\ell 2\bw(\xi,X_N)(\cdot)$ defined in the Introduction.

We discussed above two important recovery characteristics -- linear recovery error $\varrho_m(\bF,L_2)$
and sparse recovery error -- $\varrho^{ls}_{m,v}(\bF,\D_N,L_2)$. There are two more quantities, which are popular in the theory of optimal recovery. The following modification of the linear recovery procedure is also of interest. We now allow any mapping $\Phi : \bbC^m \to X_N \subset L_p(\Omega,\mu)$ where $X_N$ is a linear subspace of dimension $N\le m$ and define
$$
\varrho_m^*(\bF,L_p) := \inf_{\Phi; \xi; X_N, N\le m} \sup_{f\in \bF}\|f-\Phi(f(\xi^1),\dots,f(\xi^m))\|_p.
$$
We have the following obvious inequalities
\be\label{D1}
d_m (\bF, L_p)\le \varrho_m^*(\bF,L_p)\le \varrho_m(\bF,L_p).
\ee
One more quantity is the optimal sampling recovery. We now allow any mapping $\Phi : \bbC^m \to   L_p(\Omega,\mu)$  and define
$$
\varrho_m^o(\bF,L_p) := \inf_{\Phi; \xi } \sup_{f\in \bF}\|f-\Phi(f(\xi^1),\dots,f(\xi^m))\|_p.
$$
Here, index {\it o} stands for optimal.

{\bf Comment \ref{D}.1.} The characteristics $\varrho_m$, $\varrho_m^*$, and $\varrho_m^o$  are well studied for many particular classes of functions. For an exposition of known results we refer to the books 
\cite{TWW}, \cite{NoLN}, \cite{DTU}, \cite{VTbookMA}, \cite{NW1}--\cite{NW3} and references therein. The characteristics $\varrho_m^*$ and $\varrho_m$ are inspired by the concepts of 
the Kolmogorov width and the linear width. Probably, $\varrho_m^*$ was introduced in \cite{Dinh},  $\varrho_m$ in \cite{VT51}, and $\varrho_m^o$ in \cite{TWW}.

Let us make a brief comparison of two nonlinear recovery characteristics $\varrho^{ls}_{m,v}$ and
$\varrho_m^o$. The authors of \cite{JUV} (see Theorem 4.13 there) proved the following interesting bound for $1<p<2$, $r>1/p$
\begin{equation}\label{D2}
\varrho_{m}^{o}(\bW^{r}_p,L_2(\T^d)) \le C(r,d,p)  v^{-r+1/p-1/2} (\log v)^{(d-1)(r+1-2/p)+1/2}      
\end{equation}
provided $m\ge c(d)  v(\log(2v))^4$. It is very surprising that the bound (\ref{D2}) is exactly the same
as the bound (\ref{C6}) from Theorem \ref{CT3} for the \newline $\varrho_{m,v}^{ls}(\bW^{r}_p,\Phi_N,L_2(\T^d))$. First of all, these bounds were proved by different methods. Bound (\ref{D2}) 
was proved in \cite{JUV} with the help of a powerful result from compressed sensing theory 
and bound (\ref{C6}) is based on the result on universal discretization -- Theorem \ref{BT1}. Second, Theorem \ref{CT3} provides a sparse ($v$-term) with respect to the trigonometric system approximant while (\ref{D2}) does not provide a sparse one. Third, we cannot derive 
(\ref{C6}) from (\ref{D2}) and vise versa. The problem with deriving (\ref{D2}) from (\ref{C6}) is 
as follows. In the proof of (\ref{C6}) we use the algorithm $LS(\xi,\Sigma_v(\Phi_N))(f)$ defined
in (\ref{I4}). In addition to function values at $m$ points this algorithm compares the norms
$\|f-LS(\xi,L)(f)\|_2$, $L\in \Sigma_v(\Phi_N)$, in order to choose the smallest one. We don't know if this step 
can be performed by an algorithm based exclusively on the function values at $m$ points. 
Basically, it is a question of realizing (approximately realizing) the best $v$-term approximation by 
a simple algorithm. It is known that in many cases it can be done by an appropriate greedy type
algorithm with respect to the corresponding dictionary. The reader can find such results and their discussion in \cite{VTbookMA}, Chapter 8. 

In the case of the use of the least squares algorithms we defined in \cite{VT183} the following 
characteristic
$$
\varrho_m^{ls}(\bF,L_2) := \inf_{\xi,\,X_N} \sup_{f\in \bF} \|f-\ell 2 \bw_m(\xi,X_N)(f)\|_2=\inf_{\xi,\,X_N} \sup_{f\in \bF} \|f-LS(\xi,X_N)(f)\|_2 .
$$
We proved in \cite{VT183} an analog of Theorem \ref{IT1} with the Kolmogorov $n$-width replaced by the constrained Kolmogorov $n$-width. We now formulate the corresponding result.

{\bf Condition E($t$).} We say that an orthonormal system $\{u_i(\bx)\}_{i=1}^N$ defined on $\Omega$ satisfies Condition E($t$) with a constant $t$ if for all $\bx\in \Omega$
$$
 \sum_{i=1}^N |u_i(\bx)|^2 \le Nt^2.
$$
We now define $E(t)$-conditioned Kolmogorov width
$$
d_n^{E(t)}(\bF,L_p) :=     \inf_{\{u_1,\dots,u_n\}\, \text{satisfies Condition} E(t)}   \sup_{f\in \bF}\inf_{c_1,\dots,c_n}\|f-\sum_{i=1}^n c_iu_i\|_p.
$$

The following theorem is from \cite{VT183}.
\begin{Theorem}\label{IT1ls}  There exist two positive constants $b$ and $B$, which may depend on  $t$, such that for any   compact subset $\Omega$  of $\R^d$, any probability measure $\mu$ on it, and any compact subset $\bF$ of $\C(\Omega)$ we have
$$
\varrho_{bn}^{ls}(\bF,L_2(\Omega,\mu)) \le Bd_n^{E(t)}(\bF,L_\infty).
$$
\end{Theorem}
 
We pointed out in the Introduction that recently an outstanding progress has been done in the sampling recovery in the $L_2$ norm. We give a very brief comment on those interesting results. For special sets $\bF$ (in the reproducing kernel Hilbert space setting) the following inequality was proved (see \cite{DKU}, \cite{NSU}, \cite{KU}, and \cite{KU2}):
\be\label{R2}
\ro_{cn}(\bF,L_2) \le \left(\frac{1}{n}\sum_{k\ge n} d_k (\bF, L_2)^2\right)^{1/2}
\ee
with an absolute constant $c>0$.  

We refer the reader for results on the sampling recovery in the uniform norm to the   papers \cite{PU} and \cite{KKT}.  
The reader can find a   discussion of these results in \cite{KKLT}, Section 2.5.

\section*{Acknowledgement}
The authors would like to thank the referees for  careful reading of the paper and for helpful  suggestions and comments.

  \Addresses

\end{document}